\documentclass[11pt]{article}

\usepackage{amssymb,amsfonts,amsthm}
\usepackage{amsmath}
\usepackage{amscd}
\usepackage{url}

\usepackage{mathptmx}      

\usepackage{cite}
\makeatletter
\newcommand{\citecomment}[2][]{\citen{#2}#1\citevar}
\newcommand{\citeone}[1]{\citecomment{#1}}
\newcommand{\citetwo}[2][]{\citecomment[,~#1]{#2}}
\newcommand{\citevar}{\@ifnextchar\bgroup{;~\citeone}{\@ifnextchar[{;~\citetwo}{]}}}
\newcommand{\citefirst}{\@ifnextchar\bgroup{\citeone}{\@ifnextchar[{\citetwo}{]}}}

\makeatother

\usepackage{tikz}
\usetikzlibrary{arrows}
\usetikzlibrary{arrows,shapes}

\DeclareSymbolFont{rsfscript}{OMS}{rsfs}{m}{n}
\DeclareSymbolFontAlphabet{\mathrsfs}{rsfscript}

\numberwithin{equation}{section}

\DeclareMathOperator{\alf}{alph}
\DeclareMathOperator{\Id}{Eq}
\DeclareMathOperator{\dom}{dom}
\DeclareMathOperator{\var}{var}

\theoremstyle{plain}
\newtheorem{theorem}{Theorem}[section]

\newtheorem{cor}[theorem]{Corollary}
\newtheorem{proposition}[theorem]{Proposition}
\newtheorem{lemma}[theorem]{Lemma}

\theoremstyle{remark}
\newtheorem{remark}[theorem]{Remark}
\newtheorem{step}{Step}

\newcommand{\mJ}{\mathrsfs{J}}
\newcommand{\mR}{\mathrsfs{R}}
\newcommand{\mL}{\mathrsfs{L}}

\def\fb{finitely based}
\def\ib{identity basis}
\def\nfb{non\-finitely based}
\newcommand{\pat}{partial transformation}
\newcommand{\pts}{partial transformations}

\begin{document}

\title{Catalan monoids inherently nonfinitely based relative to finite $\mR$-trivial semigroups}
\author{O. B. Sapir \and M. V. Volkov\thanks{Institute of Natural Sciences and Mathematics, Ural Federal University, 620000 Ekaterinburg, Russia}}

\date{}

\maketitle

\begin{abstract}
We show that the 42-element monoid of all partial order preserving and extensive injections on the 4-element chain is not contained in any variety generated by a finitely based finite $\mR$-trivial semigroup. This provides unified proofs for several known facts and leads to a bunch of new results on the Finite Basis Problem for finite $\mR$- and $\mJ$-trivial semigroups.
\end{abstract}

\section{The Finite Basis Problem}
\label{sec:intro1}

The present paper develops a novel approach to the Finite Basis Problem for finite semigroups and applies it to some $\mR$- and $\mJ$-trivial monoids relevant to formal languages and representation theory. We need relatively many prerequisites from different areas before stating and proving our main result in Section~\ref{sec:main1} and proceeding with its applications in Section~\ref{sec:applications}. In this section, we provide a quick introduction to the concepts related to identity bases, while the next section collects necessary information about finite $\mR$- and $\mJ$-trivial monoids and their identities.

The idea of a finite identity basis is intuitively clear. A formal framework needed to reason about this idea in precise way is provided by equational logic; see, e.g., \cite[Chapter~II]{BuSa81}. For the reader's convenience, we recall the basics of equational logic in a form adapted to the use in this paper, that is, specialized to semigroups. When doing so, we also set up our notation.

A (\emph{semigroup}) \emph{word} is a finite sequence of symbols, called \emph{variables}. Sometimes it is convenient to use the \emph{empty word}, that is, the empty sequence. Whenever words under consideration are allowed to be empty, we always say it explicitly.

We denote words by lowercase boldface letters. If $\mathbf{w}=x_1\cdots x_k$ where $x_1,\dots,x_k$ are variables, then the set $\{x_1,\dots,x_k\}$ is denoted by $\alf(\mathbf{w})$ and the number $k$ is called the \emph{length} of $\mathbf{w}$. If $\mathbf{w}$ is the empty word, then $\alf(\mathbf{w})=\varnothing$.

Words are multiplied by concatenation, that is, for any words $\mathbf{w}'$ and $\mathbf{w}$, the sequence $\mathbf{ww}'$ is obtained by appending the sequence $\mathbf{w}'$ to the sequence $\mathbf{w}$.

Any map $\varphi\colon\alf(\mathbf{w})\to S$, where $S$ is a semigroup, is called a \emph{substitution}. The \emph{value} $\varphi(\mathbf{w})$ of $\mathbf{w}$ under $\varphi$ is the element of $S$ that results from substituting $\varphi(x)$ for each variable $x\in\alf(\mathbf{w})$ and computing the product in $S$.

A (\emph{semigroup}) \emph{identity} is a pair of words written as a formal equality. We use the sign $\bumpeq$ when writing identities (so that a pair $(\mathbf{w},\mathbf{w}')$, say, is written as $\mathbf{w}\bumpeq\mathbf{w}'$), saving the standard sign $=$ for `genuine' equalities. A semigroup $S$ \emph{satisfies} $\mathbf{w}\bumpeq \mathbf{w}'$ (or $\mathbf{w}\bumpeq \mathbf{w}'$ \emph{holds} in $S$) if $\varphi(\mathbf{w})=\varphi(\mathbf{w}')$ for every substitution $\varphi\colon\alf(\mathbf{ww}')\to S$, that is, substitutions of~elements from $S$ for the variables occurring in $\mathbf{w}$ or $\mathbf{w}'$ yield equal values to these words. For a semigroup $S$, we denote by $\Id S$ its \emph{equational theory}, that is, the set of all identities $S$ satisfies.

Given any set $\Sigma$ of identities, we say that an identity $\mathbf{w}\bumpeq\mathbf{w}'$ \emph{follows} from $\Sigma$ or that $\Sigma$ \emph{implies} $\mathbf{w}\bumpeq\mathbf{w}'$ if every semigroup satisfying all
identities in $\Sigma$ satisfies the identity $\mathbf{w}\bumpeq\mathbf{w}'$ as well. Birkhoff's completeness theorem of equational logic \cite[Theorem 14.17]{BuSa81} shows that this notion (which we have given a semantic definition) is captured by a very transparent set of inference rules, namely, substituting a word for each occurrence of a variable in an identity, multiplying an identity through on the right or the left by a word, and using symmetry and transitivity of equality.

Given a semigroup ${S}$, an \emph{\ib} for ${S}$ is any set $\Sigma\subseteq\Id{S}$ such that every identity in $\Id{S}$ follows from $\Sigma$. A semigroup ${S}$ is said to be \emph{\fb} if it possesses a finite \ib, that is, the equational theory of $S$ is finitely axiomatized; otherwise, $S$ is called \emph{\nfb}.

The \emph{Finite Basis Problem} (FBP) for a class $\mathbf{C}$ of semigroups is the question of classifying semigroups in $\mathbf{C}$ for being finitely or \nfb. Whenever the class $\mathbf{C}$ consists of finite semigroups, one may consider the FBP for $\mathbf{C}$ as an algorithmic problem, asking for an algorithm that, given (an effective description of) a semigroup $S\in\mathbf{C}$, decides whether or not $S$ is \fb. The formulation of the FBP as a decision problem is due to Tarski~\cite{Ta68} who suggested it in the 1960s in the most general setting, that is, for the class of all finite algebras. In this generality, Tarski's problem was solved by McKenzie \cite{Mc96} who proved that no algorithm can recognize which finite algebras are \fb. When restricted to finite semigroups, Tarski's problem remains open.

While partial results on the FBP for finite semigroups are numerous, they all employ only a handful of methods; see the second-named author's survey~\cite{Vo01} for a classification and analysis of these methods. One of the most powerful and easy-to-use approaches is based on the concept of an inherently \nfb\ semigroup that we explain next, after recalling the notion of a variety.

The class of all semigroups satisfying all identities from a given set $\Sigma$ is called the \emph{variety defined by $\Sigma$}. It is easy to see that the satisfaction of an identity is inherited by forming direct products and taking \emph{divisors} (that is, homomorphic images of subsemigroups) of semigroups so that each variety is closed under these two operators. In fact, varieties can be characterized by this closure property (the HSP-theorem; see \cite[Theorem 11.9]{BuSa81}).

A variety is \emph{\fb} if it can be defined by a finite set of identities; otherwise it is \emph{\nfb}. Given a semigroup ${S}$, the variety defined by $\Id S$ is denoted by $\var S$ and called the \emph{variety generated by $S$}. By the very definition, $S$ and $\var S$ are simultaneously finitely or \nfb.

A variety is said to be \emph{locally finite} if each of its finitely generated members is finite. A  finite semigroup is called \emph{inherently \nfb} if it is not contained in any finitely based locally finite variety. The variety generated by a finite semigroup is locally finite (this is an easy byproduct of the proof of the HSP-theorem; see \cite[Theorem 10.16]{BuSa81}); hence, to prove that a given finite semigroup $S$ is \nfb, it suffices to exhibit an inherently \nfb\ semigroup in the variety $\var S$.

For the argument of the preceding paragraph to be applicable, one needs some supply of inherently \nfb\ semigroups. In fact, it was not clear whether such semigroups exist\footnote{For instance, no inherently \nfb\ objects exist in several natural classes of \emph{unary semigroups}, that is, semigroups equipped with an extra unary operation \cite{Sa93,Do10}.} until Mark Sapir~\cite{Sa87a} found the first examples. In~\cite{Sa87b} he gave a structural characterization of all inherently \nfb\ semigroups that, in particular, led to an algorithm to recognize whether or not a given finite semigroup is inherently \nfb. (This contrasts McKenzie's result \cite{Mc96} that no such algorithm exists for general finite algebras.) The characterization allows one to locate many inherently \nfb\ semigroups of importance \cite{Ja02,VoGo03}; on the other hand, it reveals some limitations of the described approach to the FBP, implying that certain interesting classes of finite semigroups lack inherently \nfb\ members. The present paper aims to overcome these limitations for one of such classes.

\section{$\mR$- and $\mJ$-trivial monoids and their identities}
\label{sec:intro2}

In this paper, we focus on two classes of semigroups which only sparsely show up in the standard textbooks on semigroup theory. That is why we provide rather a self-contained introduction to these classes, relying only on few basic notions that all can be found in \cite[Chapter 1]{ClPr61} or \cite[Chapter 1]{Ho95}.

A semigroup $S$ is $\mR$-\emph{trivial} if every principal right ideal of $S$ has a unique generator. This means that the following implication holds for all $a,b\in S$:
\begin{equation}\label{eq:rtriv}
aS\cup\{a\}=bS\cup\{b\}\to a=b.
\end{equation}
A semigroup $S$ is $\mJ$-\emph{trivial} if every principal ideal of $S$ has a unique generator. In other words, $S$ is $\mJ$-trivial if the following implication holds:
\begin{equation}\label{eq:jtriv}
SaS\cup Sa\cup aS\cup\{a\}=SbS\cup Sb\cup bS\cup\{b\}\to a=b.
\end{equation}
Each $\mJ$-trivial semigroup is $\mR$-trivial. Indeed,
\[
SaS\cup Sa\cup aS\cup\{a\}=S(aS\cup\{a\})\cup aS\cup\{a\},
\]
whence the premise of the implication \eqref{eq:rtriv} implies that of the implication \eqref{eq:jtriv}. Therefore, \eqref{eq:rtriv} holds whenever \eqref{eq:jtriv} does.

A \emph{monoid} is a semigroup with an identity element. Now we introduce three series of $\mR$- and $\mJ$-trivial monoids that play a role in this paper. Let $[m]$ stand for the set of the first $m$ positive integers ordered in the usual way: $1<2<\dots<m$. By a \emph{\pat} of $[m]$ we mean an arbitrary map $\alpha$ from a subset of $[m]$ (called the \emph{domain} of $\alpha$ and denoted $\dom\alpha$) to $[m]$. A \pat\ of $[m]$ is said to be \emph{total} if its domain is the whole set $[m]$. We write \pts\ on the right of their arguments. A \pat\ $\alpha$ is \emph{order preserving} if $i\le j$ implies $i\alpha\le j\alpha$ for all $i,j\in\dom\alpha$, and \emph{extensive} if $i\le i\alpha$ for every $i\in\dom\alpha$.  Clearly, if two transformations have either of the properties of being total, order preserving, or extensive, then so does their product, and the identity transformation enjoys all three properties. Hence, the set of all total extensive transformations of $[m]$ forms a monoid which we denote by $E_m$, and the set of all order preserving transformations in $E_m$ forms a submonoid denoted by $C_{m}$ and called the \emph{Catalan monoid}. (The name comes from the cardinality of $C_m$ coinciding with the $m$-th Catalan number $\frac1{m+1}\binom{2m}m$; see \cite[Theorem 3.1]{Hi93}, \cite[Proposition 3.4]{So96}, or \cite[Theorem 14.2.8(i)]{GaMa09}.) The third series we need consists of monoids that we denote by $IC_m$ and call $i$-\emph{Catalan monoids}. Both `I' and `$i$' in the name mean `injective' and indicate that the monoid $IC_m$ is the set of all partial injections of $[m]$ that are extensive and order preserving. The `Catalan' part of the name again refers to the cardinality of the monoid: $|IC_m|$ is the $(m+1)$-th Catalan number; see \cite[Theorem 14.2.8(ii)]{GaMa09}\footnote{The proof of this result in \cite{GaMa09} relies on a well-known recurrence for Catalan numbers. In Appendix A we exhibit a direct bijection between $IC_m$  and $C_{m+1}$.}.

For each $m$, the monoid $E_m$ is $\mR$-trivial while the monoid $C_m$ is $\mJ$-trivial; see \cite[Propositions IV.3.1 and IV.1.8]{Pin86}. The monoid $IC_m$ is $\mJ$-triv\-i\-al and has commuting idempotents (that is, $IC_m$ satisfies $e^2=e\ \&\ f^2=f\to ef=fe$); see, e.g., \cite[p.88]{Hi99}.  The three series of examples are representative for the corresponding classes of finite monoids in the following sense.
\begin{proposition}
\label{prop:embedding&division}
\emph{(a)} Every finite $\mR$-trivial monoid with $m$ elements is isomorphic to a submonoid of the monoid $E_m$.\\
\emph{(b)} Every finite $\mJ$-trivial monoid is a divisor of the monoid $C_m$ for some $m$.\\
\emph{(c)} Every finite $\mJ$-trivial monoid with commuting idempotents is a divisor of the monoid $IC_m$ for some $m$.
\end{proposition}

\begin{proof}
Claims (a) and (b) can be found in \cite[Theorems IV.3.6 and IV.1.10]{Pin86} and claim (c) readily follows from  \cite[Theorem 4.1]{AHP90}.
\end{proof}

\begin{remark}
\label{rem:cayley}
Proposition~\ref{prop:embedding&division} looks quite innocent as it is stated in purely semi\-group-theoretical terms and very much resembles the textbook Cayley-type theorem that an arbitrary semigroup embeds into the monoid of all transformations of a suitable set. This analogy indeed works for claim (a), but the situation with claims (b) and (c) is very different. No direct semigroup-theoretical proof of Proposition~\ref{prop:embedding&division}(b) is known. The cited proof in~\cite{Pin86} uses a technique of Straubing~\cite{St80} which crucially depends on Simon's theorem~\cite{Si72,Si75}, a deep combinatorial result in the theory of recognizable languages. Moreover, it can be shown relatively easily that Proposition~\ref{prop:embedding&division}(b) and Simon's theorem are equivalent to each other, and therefore, a~direct proof of the former would provide a new algebraic proof of the latter. In the literature, there are many proofs of Simon's theorem, based on different approaches, but none of the proofs are purely algebraic. Similarly, Proposition~\ref{prop:embedding&division}(c) is a consequence of another deep combinatorial result due to Ash~\cite{Ash87} that solved a problem stemming from language theory; see \cite{MaPi87}.
\end{remark}

Simon's theorem mentioned in Remark~\ref{rem:cayley} establishes a strong relationship between finite $\mJ$-trivial monoids and so-called piecewise testable languages. Finite $\mR$-trivial monoids are related to a language class characterized by Eilenberg (see \cite[Theorem IV.3.3]{Pin86}) and, in a different way, by Brzozowski and Fich~\cite{BrFi80}. Both $\mR$- and $\mJ$-trivial finite monoids are of major interest for representation theory; see, e.g., \cite[Chapter 17]{St16} and references therein. A striking application of finite $\mR$-triv\-i\-al monoids to the analysis of Markov chains appears in \cite{ASST15}. Recent connections of finite $\mJ$-triv\-i\-al monoids include tropical geometry (the gossip monoid of~\cite{BDF15}) and combinatorics of Young tableaux (the stylic monoid of~\cite{AbRe22}). These diverse connections and applications make it worthwhile to study both the class $\mathbf{R}$  of finite $\mR$-trivial monoids and the class $\mathbf{J}$ of finite $\mJ$-trivial monoids. In spite of the word `trivial' present in their names, these objects are by no means trivial.

The non-triviality just noticed manifests in the study of the FBP for $\mathbf{R}$ and $\mathbf{J}$, and even for the smaller class $\mathbf{J}\cap\mathbf{Ecom}$ of finite $\mJ$-trivial monoids with commuting idempotents.  Already one of the two first examples of \nfb\ finite semigroups from Perkins's pioneering paper~\cite{Pe69} was a monoid from $\mathbf{J}\cap\mathbf{Ecom}$. Mark Sapir suggested to investigate the FBP for a certain subclass of $\mathbf{J}\cap\mathbf{Ecom}$ (containing the aforementioned example from~\cite{Pe69}); see~\cite[Problem 4.1]{Vo01} or~\cite[Problem 3.10.10]{Sa14}. This inspired massive studies by the first-named author and Jackson \cite{Sa97,Sa00,Ja99,JaSa00,Ja01} that revealed that the complexity of the FBP for this particular subclass of $\mathbf{J}\cap\mathbf{Ecom}$ is already well comparable with that for the whole class of finite semigroups.

Among numerous partial results on the FBP for $\mathbf{R}$, $\mathbf{J}$, and $\mathbf{J}\cap\mathbf{Ecom}$, we include here only the following concerning the series $\{E_m\}_{m\ge1}$, $\{C_m\}_{m\ge1}$, and $\{IC_m\}_{m\ge1}$:
\begin{proposition}
\label{prop:FBP}
\emph{(a)} The monoid $E_m$ is \fb\ if and only if $m\le 4$.\\
\emph{(b)} The monoid $C_m$ is \fb\ if and only if $m\le 4$.\\
\emph{(c)}  The monoid $IC_m$ is \fb\  if and only if $m<3$.
\end{proposition}

\begin{proof}
Claim (a) is a combination of \cite[Theorem 1.1]{Go07}, \cite[Proposition 3.3]{Lee09}, and \cite[Theorem 1]{LiLo10}. The three cited statements deal with the cases $m>4$, $m<4$, and $m=4$, respectively.

Claim (b) is a part of \cite[Theorem 1]{Vo04}.

Claim (c) is a combination of \cite[Theorem 1]{Go06a}, \cite[Proposition 3.1(i)]{Ed77}, and \cite[Theorem 4.4]{HCL15}. The three cited statements deal with the cases $m>3$, $m<3$, and $m=3$, respectively.
\end{proof}

Proposition~\ref{prop:FBP}(b) in deduced in \cite{Vo04} from results by Blanchet-Sadri~\cite{Bl93,Bl94} combined with a description of the equational theory of the monoid $C_{m}$. We recall the description as it is utilized in this paper too.

A word $\mathbf{u}=x_1\cdots x_k$, where $x_1,\dots,x_k$ are variables, is a \emph{scattered subword} of a word $\mathbf{v}$ if there are words $\mathbf{v}_0, \mathbf{v}_1, \dots,  \mathbf{v}_{k-1}, \mathbf{v}_k$ (some of which may be empty) with
\begin{equation}
\label{eq:scattered}
\mathbf{v} = \mathbf{v}_0 x_1 \mathbf{v}_1\cdots \mathbf{v}_{k-1}x_k\mathbf{v}_k.
\end{equation}
Thus, \eqref{eq:scattered} means that $\mathbf{u}$ as a sequence of variables is a subsequence in $\mathbf{v}$. For $m\ge 1$, denote by $J_m$ the set of all identities $\mathbf{w}\bumpeq\mathbf{w'}$ such that the words $\mathbf{w}$ and $\mathbf{w'}$ have the same scattered subwords of length $\le m$. For convenience, let $J_0$ denote the set of all semigroup identities.

\begin{proposition}[{\!\mdseries\cite[Theorem 2]{Vo04}}]
\label{prop:JMC}
$\Id C_{m}=J_{m-1}$ for each $m\ge1$.
\end{proposition}

A description of the equational theory of the monoid $E_{m}$ is also known (see \cite[Proposition 2.2]{Go07}), but we do not reproduce it here as it is not used in this paper. However, we need a description of $\Id IC_m$. It  involves the following notion: a scattered subword $\mathbf{u}=x_1\cdots x_k$ of a word $\mathbf{v}$ is said to be \emph{unambiguously scattered} if $\mathbf{v}$ has a unique decomposition of the form \eqref{eq:scattered}. For each $m\ge0$, denote by $U_m$ the set of all identities $\mathbf{w}\bumpeq\mathbf{w'}$ such that
\begin{itemize}
  \item[(i)] $\alf(\mathbf{w})=\alf(\mathbf{w}')$;
  \item[(ii)] $\mathbf{w}$ and $\mathbf{w'}$ have the same unambiguously scattered subwords of length $\le m$;
  \item[(iii)] if $\mathbf{u}=x_1\cdots x_k$ with $k\le m$ is unambiguously scattered in $\mathbf{w}$ and $\mathbf{w'}$ and
\begin{align*}
\mathbf{w}&= \mathbf{w}_0 x_1 \mathbf{w}_1\cdots \mathbf{w}_{k-1}x_k\mathbf{w}_k,\\
\mathbf{w}'&= \mathbf{w}'_0 x_1 \mathbf{w}'_1\cdots \mathbf{w}'_{k-1}x_k\mathbf{w}'_k,
\end{align*}
then $\alf(\mathbf{w}_i)=\alf(\mathbf{w}'_i)$ for all $i=0,1,\dots,k$.
\end{itemize}
Notice that for $m=0$, the conditions (ii) and (iii) become void so that $U_0$ is merely the set of all identities $\mathbf{w}\bumpeq\mathbf{w'}$ satisfying $\alf(\mathbf{w})=\alf(\mathbf{w}')$.

\begin{proposition}[{\!\mdseries\cite[Proposition 2]{Go06a}}]
\label{prop:UMPIC}
$\Id IC_m=U_{m-1}$ for each $m\ge1$.
\end{proposition}

The identity sets $J_m$ and $U_m$ relate as follows:
\begin{lemma}
\label{lem:inclusion}
$J_{m+1}\subseteq U_m$ for each $m\ge0$, and for $m>0$, the inclusion is strict.
\end{lemma}

\begin{proof}
First consider the case $m=0$. The fact that two words $\mathbf{w}$ and $\mathbf{w}'$ have the same scattered subwords of length 1 amounts to saying that $\mathbf{w}$ and $\mathbf{w}'$ involve the same variables, that is, $\alf(\mathbf{w})=\alf(\mathbf{w}')$. Thus, $J_1$ coincides with the set of all identities $\mathbf{w}\bumpeq\mathbf{w}'$ satisfying $\alf(\mathbf{w})=\alf(\mathbf{w}')$, and as observed after the definition of the set $U_m$,
the same holds for $U_0$. Thus, $J_1=U_0$.

Now let $m>0$. To prove that $J_{m+1}\subseteq U_m$, we take an arbitrary identity $\mathbf{w}\bumpeq\mathbf{w}'$ from $J_{m+1}$ and show that it lies in $U_m$, arguing by contradiction. If $(\mathbf{w}\bumpeq\mathbf{w}')\notin U_m$, then the identity violates one of the conditions (i)--(iii) from the definition of $U_m$. As already observed, the fact that $\mathbf{w}$ and $\mathbf{w}'$ share scattered subwords of length 1 implies $\alf(\mathbf{w})=\alf(\mathbf{w}')$ so the condition (i) holds for $\mathbf{w}\bumpeq\mathbf{w}'$.

Suppose that the condition (ii) fails, that is, for some $k\le m$, one of the words $\mathbf{w}$ or $\mathbf{w}'$ has an unambiguously scattered subword $\mathbf{u}$ of length $k$ which is not unambiguously scattered in the other word. Let, for certainty, $\mathbf{u}$ be unambiguously scattered in $\mathbf{w}'$ but not in $\mathbf{w}$. As $\mathbf{w}$ and $\mathbf{w'}$ share scattered subwords of length $k\le m+1$, the word $\mathbf{u}$ does occur as a scattered subword in $\mathbf{w}$ but not in a unique way. Let $\mathbf{u}=x_1\cdots x_k$, where $x_1,\dots,x_k$ are variables. Write the word $\mathbf{w}$ as $\mathbf{w}=\mathbf{w}_0 x_1 \mathbf{v}_1$ with $x_1\notin\alf(\mathbf{w}_0)$ so that the designated occurrence of the variable $x_1$ is  the leftmost occurrence of this variable in $\mathbf{w}$. Then write $\mathbf{v}_1$ as $\mathbf{v}_1=\mathbf{w}_1 x_2 \mathbf{v}_2$ with $x_2\notin\alf(\mathbf{w}_1)$, and so on. After $k$ steps, we get the following $k$ decompositions:
\begin{gather*}
\begin{aligned}
\mathbf{w}&=\mathbf{w}_0 x_1 \mathbf{v}_1,&\quad x_1\notin&\alf(\mathbf{w}_0),\\
\mathbf{v}_1&=\mathbf{w}_1 x_2 \mathbf{v}_2,&\quad x_2\notin&\alf(\mathbf{w}_1),\\
\mathbf{v}_2&=\mathbf{w}_2 x_3 \mathbf{v}_3,&\quad x_3\notin&\alf(\mathbf{w}_2),\\
            &\ \ \vdots                    &               &\vdots\\
\mathbf{v}_{k-1}&=\mathbf{w}_{k-1} x_k \mathbf{v}_k,&\quad x_k\notin&\alf(\mathbf{w}_{k-1}).
\end{aligned}
\end{gather*}
Combining these equalities and renaming $\mathbf{v}_k$ into $\mathbf{w}_k$, we decompose $\mathbf{w}$ as follows:
\begin{equation}
\label{eq:firstocc}
\mathbf{w}= \mathbf{w}_0 x_1 \mathbf{w}_1\cdots \mathbf{w}_{k-1}x_k\mathbf{w}_k,
\end{equation}
where $x_i\notin\alf(\mathbf{w}_{i-1})$ for all $i=1,\dots,k$.

Denote by $p_i$ the number of the position occupied by the variable $x_i$ in the representation \eqref{eq:firstocc}. We have assumed that $\mathbf{u}$ occurs as a scattered subword of $\mathbf{w}$ also in a way different from \eqref{eq:firstocc}. Fix such an alternative occurrence of $\mathbf{u}$ and denote by $q_i$ the number of the position occupied by the variable $x_i$ in the representation
\begin{equation}
\label{eq:altocc}
\mathbf{w}= \mathbf{w}'_0 x_1 \mathbf{w}'_1\cdots \mathbf{w}'_{k-1}x_k\mathbf{w}'_k,
\end{equation}
corresponding to this alternative occurrence. By the definition, $p_1<\cdots<p_k$ and $q_1<\dots<q_k$. Our construction of \eqref{eq:firstocc} ensures that $p_i\le q_i$ for all $i=1,\dots,k$, and for some $j$, we have $p_j<q_j$ since \eqref{eq:altocc} and \eqref{eq:firstocc} differ.

Using backward induction on $j$, we prove that for some $i$, the variable $x_i$ occurs in the word $\mathbf{w}_i$. Indeed, for $j=k$ the inequality $p_k<q_k$ implies that $x_k$ occurs in the word $\mathbf{w}_k$. Suppose that $j<k$. If $x_j$ occurs in $w_j$, then our claim holds. Otherwise, $p_{j+1}\le q_j<q_{j+1}$ and the induction assumption applies.

Fix an $i$ such that $x_i$ occurs in $\mathbf{w}_i$, that is, between the occurrences of $x_i$ and $x_{i+1}$ designated in \eqref{eq:firstocc} or, if $i=k$, after the occurrence of $x_k$ designated in \eqref{eq:firstocc}.
This means that the word $x_1\cdots  x_ix_ix_{i+1}\cdots x_k$ (or the word $x_1\cdots x_kx_k$ if $i=k$) of length $k+1$ is a scattered subword in $\mathbf{w}$. Since the identity $\mathbf{w}\bumpeq\mathbf{w'}$ lies in $J_{m+1}$,  the words $\mathbf{w}$ and $\mathbf{w}'$ share scattered subwords of length $k+1\le m+1$. Hence, $x_1\cdots x_ix_ix_{i+1}\cdots x_k$ (or $x_1\cdots x_kx_k$ if $i=k$) is a scattered subword in $\mathbf{w}'$. However, the word $\mathbf{u}=x_1\cdots x_k$ is not unambiguously scattered in $x_1\cdots x_ix_ix_{i+1}\cdots x_k$ (nor in $x_1\cdots x_kx_k$), whence $\mathbf{u}$ is not unambiguously scattered in $\mathbf{w}'$, a contradiction.

It remains to consider the case where the condition (ii) holds for $\mathbf{w}\bumpeq\mathbf{w}'$, but the condition (iii) fails. Then $\mathbf{w}$ and $\mathbf{w}'$ share unambiguously scattered subwords of length up to $m$, but for some $k\le m$ and some unambiguously scattered subword $\mathbf{u}=x_1\cdots x_k$ of $\mathbf{w}$ and $\mathbf{w}'$, there exists some index $i\in\{0,\dots,k\}$ for which $\alf(\mathbf{w}_i)\ne\alf(\mathbf{w}'_i)$ where $\mathbf{w}_0, \mathbf{w}_1,\dots,\mathbf{w}_k,\mathbf{w}'_0, \mathbf{w}'_1,\dots,\mathbf{w}'_k$ come from the decompositions
\begin{align*}
\mathbf{w}&= \mathbf{w}_0 x_1 \mathbf{w}_1\cdots \mathbf{w}_{k-1}x_k\mathbf{w}_k,\\
\mathbf{w}'&= \mathbf{w}'_0 x_1 \mathbf{w}'_1\cdots \mathbf{w}'_{k-1}x_k\mathbf{w}'_k.
\end{align*}
For certainty, assume that there is a variable $t\in\alf(\mathbf{w}_i)\setminus\alf(\mathbf{w}'_i)$. Then the word
\[
\mathbf{v}=\begin{cases} tx_1\cdots x_k&\text{if \ $i=0$,}\\
x_1\cdots x_itx_{i+1}\cdots x_k&\text{if \ $0<i<k$,}\\
x_1\cdots x_kt&\text{if \ $i=k$}
\end{cases}
\]
has length $k+1$ and is a scattered subword of the word $\mathbf{w}$. Recall that $\mathbf{w}$ and $\mathbf{w}'$ share scattered subwords of length $k+1\le m+1$ whence $\mathbf{v}$ occurs as a scattered subword also in $\mathbf{w}'$. Since $\mathbf{u}$ is unambiguously scattered in $\mathbf{w}'$, the positions of the variables $x_1,\dots,x_k$ in $\mathbf{w}'$ are uniquely fixed. Therefore, when $\mathbf{v}$ is scattered over $\mathbf{w}'$, the   occurrence of $t$ must happen within the part of $\mathbf{w}'$ determined by the neighbor(s) of $t$ in $\mathbf{v}$, that is, within $\mathbf{w}'_i$. This contradicts the assumption $t\notin\alf(\mathbf{w}'_i)$.

We have proved the inclusion $J_{m+1}\subseteq U_m$ for each $m\ge0$. To show that it is strict if $m>0$, consider the identity
\begin{equation}
\label{eq:idcommute}
x^{m+1}y^{m+1}\bumpeq y^{m+1}x^{m+1}.
\end{equation}
The identity belongs to the set $U_m$ since the words $x^{m+1}y^{m+1}$ and $y^{m+1}x^{m+1}$ involve the same variables and have no unambiguously scattered subwords of length $\le m$. On the other hand, the word $xy$ is a scattered subword in $x^{m+1}y^{m+1}$ but not in $y^{m+1}x^{m+1}$, whence \eqref{eq:idcommute} is not in $J_2$, and therefore, in no $J_{m+1}$ with $m>0$.
\end{proof}

In view of Propositions~\ref{prop:JMC} and~\ref{prop:UMPIC}, translating Lemma~\ref{lem:inclusion} into the language of varieties yields the following fact useful for applications of our main result:
\begin{proposition}\label{prop:inclusion}
$\var IC_1 = \var C_2$ and $\var IC_m \subset\var C_{m+1}$ for all $m \ge2$.
\end{proposition}

We mention that for $m=2$, the result of Proposition~\ref{prop:inclusion} is known; see \cite{Lee08} where the 5-element monoid isomorphic to $C_3$ appears under the name $A_0^1$ while the 5-element monoid isomorphic to $IC_2$ bears the name $B_0^1$.

The final auxiliary fact we need deals with identities of finite $\mR$-trivial monoids. It is an immediate combination of \cite[Lemma 5.2]{BrFi80} and \cite[Lemma 3]{Si75}.

\begin{proposition}\label{prop:iden}
Let $M$ be an $\mR$-trivial monoid and $|M|=m$. Then $M$ satisfies any identity $\mathbf{u}\bumpeq\mathbf{uv}$ such that the word $\mathbf{u}$ can be decomposed as $\mathbf{u} = \mathbf{u}_1 \mathbf{u}_2\cdots \mathbf{u}_m$ with
$\alf(\mathbf{u}_1) \supseteq \alf(\mathbf{u}_2)\supseteq \cdots \supseteq \alf(\mathbf{u}_m)\supseteq \alf(\mathbf{v})$.
\end{proposition}

\section{Main result}
\label{sec:main1}

Recall from Section~\ref{sec:intro1} that an inherently \nfb\ semigroup is not contained in any \fb\ variety generated by a finite semigroup. It follows from Mark Sapir's characterization of inherently \nfb\ semigroups in~\cite{Sa87b}, that no $\mJ$- or $\mR$-trivial semigroup can posses this property. Our main and only theorem is that the $i$-Catalan monoid $IC_4$ has an albeit weaker but similar feature.

\begin{theorem}
\label{thm:C5}
The $i$-Catalan monoid $IC_4$ is not contained in any \fb\ variety generated by a finite $\mR$-trivial semigroup.
\end{theorem}

We express this result by saying that the monoid $IC_4$ is \emph{inherently nonfinitely based relative to finite $\mR$-trivial semigroups}. The idea of relativizing the property of being inherently \nfb, suggested by Jackson and the second-named author~\cite{JaVo09} in the context of quasivarieties, was motivated by the fact (discovered by Margolis and Mark Sapir~\cite{MaSa95}) that every finite semigroup lies in a locally finite finitely based quasivariety. This means that if one attempts to literally transfer the notion of an inherently \nfb\ semigroup to the quasivariety setting by calling a finite semigroup $S$ \emph{inherently nonfinitely $q$-based} if $S$ is not contained in any locally finite finitely based quasivariety, then the resulting notion would be void. However, relativized versions of this notion make perfect sense, and their study in~\cite{JaVo09} led to a number of interesting results.

Back to the realm of varieties, a finite semigroup $S$ is called \emph{weakly \fb} if $S$ is not inherently \nfb. To the best of our knowledge, Theorem~\ref{thm:C5} gives the first example of a weakly \fb\ semigroup that is  inherently \nfb\ relative to a large and important class of finite semigroups.

Proving Theorem~\ref{thm:C5} amounts to showing that if $IC_4\in\var S$ where $S$ is a finite $\mR$-trivial semigroup, then $S$ is \nfb. For this, we employ a sufficient condition under which a semigroup is nonfinitely based from the first-named author's paper~\cite{Sa15}. Given a semigroup $S$, a word $\mathbf{u}$ is called an \emph{isoterm for} $S$ if the only word $\mathbf{v}$ such that $S$ satisfies the identity $\mathbf{u} \bumpeq  \mathbf{v}$ is the word $\mathbf{u}$ itself. We fix a countably infinite set $\mathfrak A$ of variables and denote by $\mathfrak A^+$ the set of all words whose variables lie in $\mathfrak A$. The set $\mathfrak A^+$ forms a semigroup under concatenation of words. We assume that all nonempty words that we encounter below come from $\mathfrak A^+$.

\begin{proposition}[{\!\mdseries\cite[Corollary 2.2]{Sa15}}]
\label{prop:nfbcor} A semigroup $S$ is \nfb\ whenever for infinitely many $n$, there exists a word $\mathbf{u}_n$ with the following properties:
\begin{itemize}
  \item[\emph{1)}] $|\alf(\mathbf{u}_n)|\ge n$ and $\mathbf{u}_n$ is not an isoterm for $S$;
  \item[\emph{2)}] if a word $\mathbf{u}$ with $|\alf(\mathbf{u})|<n$ is such that $\vartheta(\mathbf{u})=\mathbf{u}_n$ for some substitution $\vartheta\colon \mathfrak A \rightarrow \mathfrak A ^+$, then $\mathbf{u}$ is an isoterm for $S$.
\end{itemize}
\end{proposition}

We proceed with constructing a two-parameter family of words $\{\mathbf{u}_n(m)\}_{n,m\ge 1}$ that we need to apply Proposition~\ref{prop:nfbcor} to finite $\mR$-trivial semigroups. Define a map $f\colon\mathfrak A^+\rightarrow \mathfrak A ^+$ as follows. For each $\mathbf{u}\in\mathfrak A ^+$, let $f(\mathbf{u})$ be the word obtained by inserting a `fresh' variable (that is, a variable not in $\alf(\mathbf{u})$) between each pair of adjacent variables in $\mathbf{u}$ and by adding two `fresh' variables before the first and after the last variables of $\mathbf{u}$. (For example, $f(z_1 z_2 z_3) = p_0z_1 p_1 z_2 p_2 z_3 p_3$ where all variables $p_0,p_1,p_2,p_3$ are distinct.) Further, let $f^{k}(\mathbf{u})=\underbrace{f(f(\dots f}_{\text{$k$ times}}(\mathbf{u})\dots))$.

For each $n\ge1$, let $\mathbf{y}_n=y_1 y_2 \cdots y_n$, where $y_1,y_2,\dots,y_n$ are distinct variables. For each $m\ge 1$, we define
\begin{equation}
\label{eq:construct}
\mathbf{u}_n(m)=xf^{m-1}(\mathbf{y}_n)x f^{m-2}(\mathbf{y}_n) \cdots xf(\mathbf{y}_n) x\mathbf{y}_n,
\end{equation}
where the variable $x$ does not occur in the word $f^{m-1}(\mathbf{y}_n)$. Then the following two properties of the words $\mathbf{u}_n(m)$ readily follow from the construction \eqref{eq:construct}:
\begin{itemize}
  \item[(P1)] For all $y,z\in\mathfrak A$, the word $yz$ occurs in $\mathbf{u}_n(m)$ as a factor\footnote{An occurrence of a word $\mathbf{u}$ in a word $\mathbf{w}$ as a \emph{factor} is any decomposition of the form $\mathbf{w}=\mathbf{v}'\mathbf{u}\mathbf{v}''$ where the words $\mathbf{v}',\mathbf{v}''$ may be empty. If such a decomposition of $\mathbf{w}$ is unique, then we say that the factor $\mathbf{u}$ occurs in $\mathbf{w}$ once; otherwise, $\mathbf{u}$ occurs in $\mathbf{w}$ more than once.} at most once.
  \item[(P2)] For every $z\in\mathfrak A$, there are at least $n$ pairwise distinct variables between any two occurrences of $z$ in $\mathbf{u}_n(m)$.
\end{itemize}

If a variable occurs exactly once in a word $\mathbf{u}$, then the variable is called \emph{linear} in $\mathbf{u}$. If a variable occurs more than once in $\mathbf{u}$, then we say that the variable is \emph{repeated} in $\mathbf{u}$. A word $\mathbf{u}$ is called \emph{sparse} if every two occurrences of a repeated variable in $\mathbf{u}$ sandwich some linear variable.

\begin{lemma}
\label{lem:sparse}
Suppose that a word $\mathbf{u}$ with $|\alf(\mathbf{u})|<n$ is such that $\vartheta(\mathbf{u})=\mathbf{u}_n(m)$ for some $m\ge 1$ and some substitution $\vartheta\colon \mathfrak A \rightarrow \mathfrak A ^+$. Then $\mathbf{u}$ is sparse.
\end{lemma}

\begin{proof}
For every repeated variable $z$ of $\mathbf{u}$, the word $\vartheta(z)$ occurs as a factor in $\vartheta(\mathbf{u})=\mathbf{u}_n(m)$ more than once. In view of the property (P1), we see that $\vartheta(z)$ must be a single variable. Now, arguing by contradiction, suppose that the word $\mathbf{u}$ is not sparse. We choose two occurrences ${_1t}$ and ${_2t}$ of a repeated variable $t$ of $\mathbf{u}$ such that
\begin{itemize}
  \item[(a)] no linear variable occurs in $\mathbf{u}$ between ${_1t}$ and ${_2t}$, and
  \item[(b)] ${_1t}$ and ${_2t}$ are at the minimum possible distance with the property (a).
\end{itemize}
Let $\mathbf{w}$ stand for the part of word $\mathbf{u}$ formed by the variables following ${_1t}$ and preceding ${_2t}$. Then either $\mathbf{w}$ is empty or all variables in $\mathbf{w}$ are repeated in $\mathbf{u}$  because of (a), and, moreover, they are pairwise distinct because of (b). If $\mathbf{w}$ is empty, then ${_1t}$ and ${_2t}$ are adjacent in $\mathbf{u}$ whence so are the corresponding occurrences of the variable $\vartheta(t)$ in $\vartheta(\mathbf{u})=\mathbf{u}_n(m)$. This contradicts the property (P2). If $\mathbf{w}$ is nonempty, then it has less than $n$ variables since the whole $\mathbf{u}$ involves less than $n$ distinct variables and the variables of $\mathbf{w}$ are all distinct. As every variable of $\mathbf{w}$ is repeated in $\mathbf{u}$, its image under $\vartheta$ is a variable. Hence there are less than $n$ variables between the two occurrences of the variable $\vartheta(t)$ in $\vartheta(\mathbf{u})=\mathbf{u}_n(m)$ that correspond to ${_1t}$ and ${_2t}$. This again contradicts the property (P2).
\end{proof}

\begin{lemma}
\label{lem:sparseC5}
Every sparse word is an isoterm for the $i$-Catalan monoid $IC_4$.
\end{lemma}

\begin{proof}
Let $\mathbf{u}$ be a sparse word and suppose that $IC_4$ satisfies $\mathbf{u}\bumpeq\mathbf{v}$ for some $\mathbf{v}$. We have to prove that $\mathbf{v}=\mathbf{u}$.

We stepwise establish more and more similarities between the words $\mathbf{u}$ and $\mathbf{v}$, eventually showing that they coincide. In doing so, we use Proposition~\ref{prop:UMPIC} which ensures that the words $\mathbf{u}$ and $\mathbf{v}$ satisfy the conditions (i)--(iii) in the definition of the set $U_m$ with $m=3$. In particular, $\alf(\mathbf{u})=\alf(\mathbf{v})$ by the condition (i).

\begin{step}
If a variable is linear in $\mathbf{u}$, then it is so in $\mathbf{v}$, and vice versa.
\end{step}

\begin{proof}
If a variable $t$ is linear in $\mathbf{u}$, then the word $t$ of length 1 is unambiguously scattered in $\mathbf{u}$. By the condition (ii) $t$ is unambiguously scattered in $\mathbf{v}$, but this means that the variable $t$ is linear in $\mathbf{v}$. The same argument proves the converse statement.
\end{proof}

\begin{step}
The linear variables occur in $\mathbf{u}$ and $\mathbf{v}$ in the same order.
\end{step}

\begin{proof}
Take any linear variables $t_1$ and $t_2$. If $t_1$ precedes $t_2$ in $\mathbf{u}$, then the subword $t_1t_2$ of length 2 is unambiguously scattered in $\mathbf{u}$. By the condition (ii) $t_1t_2$ is unambiguously scattered in $\mathbf{v}$. Thus, $t_1$ precedes $t_2$ in $\mathbf{v}$ too. The same argument proves that if $t_1$ precedes $t_2$ in $\mathbf{v}$, then it does so in $\mathbf{u}$.
\end{proof}

From Steps 1 and 2, we get the following decompositions of $\mathbf{u}$ and $\mathbf{v}$:
\begin{align}
\mathbf{u} &= \mathbf{a}_0t_1 \mathbf{a}_1t_2 \cdots t_{k-1}\mathbf{a}_{k-1} t_k \mathbf{a}_k,\label{eq:dec1}\\
\mathbf{v} &= \mathbf{b}_0t_1 \mathbf{b}_1t_2 \cdots t_{k-1}\mathbf{b}_{k-1} t_k \mathbf{b}_k,\label{eq:dec2}
\end{align}
where  $t_1, \dots, t_k$ are the linear variables of $\mathbf{u}$ and $\mathbf{v}$ and the words $\mathbf{a}_0, \mathbf{a}_1, \dots, \mathbf{a}_{k-1},\mathbf{a}_k$ and $\mathbf{b}_0, \mathbf{b}_1, \dots, \mathbf{b}_{k-1},\mathbf{b}_k$ are either empty or involve only variables that are repeated in $\mathbf{u}$ and $\mathbf{v}$. It remains to prove that $\mathbf{a}_i = \mathbf{b}_i$ for each $i=0,1,\dots,k-1,k$. In the next steps, we fix such an index $i$. In order to treat the extreme cases when $i=0$ or $i=k$ in the same way as $i=1,\dots,k-1$, we adopt the convention that $t_0$ and $t_{k+1}$ are dummy symbols meaning the absence of a variable.

\begin{step}
Every variable occurs in the word $\mathbf{a}_i$ at most once.
\end{step}

\begin{proof}
This follows from the condition that the word $\mathbf{u}$ is sparse, combined with the fact that all variables that may occur in $\mathbf{a}_i$ are repeated in $\mathbf{u}$.
\end{proof}

\begin{step}
$\alf(\mathbf{a}_i) = \alf(\mathbf{b}_i)$.
\end{step}

\begin{proof}
This follows from the condition (iii) applied to the word $t_it_{i+1}$ of length $\le2$ which is unambiguously scattered in $\mathbf{u}$ and $\mathbf{v}$.
\end{proof}

\begin{step}
Every variable occurs in the word $\mathbf{b}_i$ at most once.
\end{step}

\begin{proof}
If a variable $z$ occurs in $\mathbf{b}_i$, then by Steps 4 and 3 it occurs in $\mathbf{a}_i$ exactly once. Then the word $t_izt_{i+1}$ of length $\le3$ is unambiguously scattered in $\mathbf{u}$. The condition (ii) yields that $t_izt_{i+1}$ is unambiguously scattered in $\mathbf{v}$, and this implies that $z$ occurs in $\mathbf{b}_i$ exactly once.
\end{proof}

\begin{step}
$\mathbf{a}_i = \mathbf{b}_i$.
\end{step}

\begin{proof}
Taking into account Steps 3--5, it remains to show that the variables forming the words $\mathbf{a}_i$ and $\mathbf{b}_i$ occur in these words in the same order. Take any variables $z_1$ and $z_2$ that occur in $\mathbf{a}_i$. Suppose that $z_1$ precedes $z_2$ in $\mathbf{a}_i$. The word $t_iz_2t_{i+1}$ of length $\le3$ is unambiguously scattered in $\mathbf{u}$. By the condition (ii) $t_iz_2t_{i+1}$  is unambiguously scattered in $\mathbf{v}$ as well. Applying to this word the condition (iii), we conclude that the same variables occur between $t_i$ and $z_2$ in $\mathbf{a}_i$ and $\mathbf{b}_i$. Since the variable $z_1$ appears between $t_i$ and $z_2$ in $\mathbf{a}_i$, it does so in $\mathbf{b}_i$. Hence, $z_1$ precedes $z_2$ in $\mathbf{b}_i$. The same argument proves that if $z_1$ precedes $z_2$ in $\mathbf{b}_i$, it does so in $\mathbf{a}_i$.
\end{proof}

From Step 6 and the decompositions \eqref{eq:dec1} and \eqref{eq:dec2}, we see that $\mathbf{u} = \mathbf{v}$.
\end{proof}

If a semigroup $T$ belongs to the variety generated by a semigroup $S$, then every identity holding in the latter semigroup also holds in the former. Therefore, every word that is an isoterm for $T$ is an isoterm for $S$ as well. By this observation, Lemma~\ref{lem:sparseC5} yields the following.
\begin{cor}
\label{cor:isoterm}
Sparse words are isoterms for any semigroup $S$ with $IC_4\in\var S$.
\end{cor}

Next, we show that for each $n$ and any given finite $\mR$-trivial semigroup $S$, the word $\mathbf{u}_n(m)$ with sufficiently large $m$ is not an isoterm for $S$.

\begin{lemma}
\label{lem:noisoterm}
If $S$ is a finite $\mR$-trivial semigroup, then for every $n\ge1$, the identity $\mathbf{u}_n(|S|+1)\bumpeq\mathbf{u}_n(|S|+1)x$ holds in $S$.
\end{lemma}

\begin{proof}
Let $M=S\cup\{1\}$ where $1$ is a fresh symbol. If the multiplication in $S$ is extended to $M$ in a unique way such that $1$ becomes the identity element, then $M$ becomes an $\mR$-trivial monoid with $|S|+1$ elements. Denoting $|S|+1$ by $m$, we apply Proposition~\ref{prop:iden} (with $\mathbf{u}_n(m)$ and $x$ in the roles of $\mathbf{u}$ and, respectively, $\mathbf{v}$) to the monoid $M$ and the identity $\mathbf{u}_n(m)\bumpeq\mathbf{u}_n(m)x$. Indeed, \eqref{eq:construct} can be viewed as the decomposition
\[
\mathbf{u}=\mathbf{u}_n(m)=\underbrace{xf^{m-1}(\mathbf{y}_n)}_{\mathbf{u}_1}\cdot \underbrace{xf^{m-2}(\mathbf{y}_n)}_{\mathbf{u}_2} \cdots \underbrace{xf(\mathbf{y}_n)}_{\mathbf{u}_{m-1}}\cdot\underbrace{x\mathbf{y}_n}_{\mathbf{u}_m}
\]
with $\alf(\mathbf{u}_1) \supseteq \alf(\mathbf{u}_2)\supseteq \cdots \supseteq \alf(\mathbf{u}_m)\supseteq \alf(x)=\{x\}$, whence the identity $\mathbf{u}_n(m)\bumpeq\mathbf{u}_n(m)x$ holds in the monoid $M$ and so in the subsemigroup $S$.
\end{proof}

\begin{proof}[Proof of Theorem~\ref{thm:C5}]
Take any finite $\mR$-trivial semigroup $S$ such that $\var S$ contains the $i$-Catalan monoid $IC_4$; we have to prove that $S$ is \nfb. For this, we show that $S$ fulfills the conditions 1) and 2) in Proposition~\ref{prop:nfbcor}, with the words $\mathbf{u}_n(|S|+1)$ defined by \eqref{eq:construct} playing the role of the words $\mathbf{u}_n$, $n=1,2,\dotsc$. Indeed, the condition 1) is satisfied since $|\alf(\mathbf{u}_n(|S|+1))|\ge n$ by the construction and $\mathbf{u}_n(|S|+1)$ is not an isoterm for $S$ by Lemma~\ref{lem:noisoterm}. The condition 2) is satisfied because by Lemma~\ref{lem:sparse} every word $\mathbf{u}$ with $|\alf(\mathbf{u})|<n$ such that $\vartheta(\mathbf{u})=\mathbf{u}_n(|S|+1)$ for some substitution $\vartheta\colon \mathfrak A \rightarrow \mathfrak A ^+$ is sparse, and by Corollary~\ref{cor:isoterm} every sparse word is an isoterm for $S$. Hence, Proposition~\ref{prop:nfbcor} ensures that $S$ is \nfb.
\end{proof}

Using Proposition~\ref{prop:inclusion}, we immediately get the following handy fact:

\begin{cor}
\label{cor:C5}
The Catalan monoid $C_5$ is not contained in any \fb\ variety generated by a finite $\mR$-trivial semigroup.
\end{cor}

\section{Applications}
\label{sec:applications}

Due to Theorem~\ref{thm:C5} and Corollary~\ref{cor:C5}, in order to prove that a finite $\mR$-trivial semigroup $S$ is \nfb, it suffices to find the $i$-Catalan monoid $IC_4$ or the Catalan monoid $C_5$ in the variety $\var S$. This provides unified proofs for many known `negative' facts on the FBP for finite $\mR$- and $\mJ$-trivial semigroups and leads to several new `negative' results.

We start with a brief overview of known results deducible from Theorem~\ref{thm:C5} or Corollary~\ref{cor:C5} and then proceed with applications to certain monoids that have been considered in the literature but not yet from the viewpoint of the FBP.

\subsection{New proofs of known facts}
\label{subsec:known}

\paragraph{1.} The `negative' parts of items (a) and (b) in Proposition~\ref{prop:FBP} claim that the monoids $E_m$ and $C_m$ are \nfb\ whenever $m\ge 5$. By the definition, $C_m$ is a submonoid of $E_m$, and it is easy to see that $C_5$ is isomorphic to a submonoid of $C_m$ for each $m\ge 5$. Hence, $C_5$ lies in both $\var C_m$ and $\var E_m$ whenever $m\ge 5$, and Corollary~\ref{cor:C5} applies. The part of Proposition~\ref{prop:FBP}(c) dealing with $m>3$ similarly follows from Theorem~\ref{thm:C5}. (Notice that Proposition~\ref{prop:FBP} was not used in the proof of Theorem~\ref{thm:C5} so that there is no \emph{circulus in probando} here.)

\paragraph{2.} Semigroups $S$ and $T$ are called \emph{equationally equivalent} if $\Id S=\Id T$, that is, $S$ and $T$ satisfy the same identities. Several series of finite $\mathrsfs{J}$-trivial monoids parameterized by positive integers appear in the literature, and in spite of arising due to completely unrelated reasons and consisting of elements of a very different nature, it often turns out that the $m$-th monoid in each series is equationally equivalent to $C_m$ (or $C_{m+1}$ if the monoids in the series are indexed by the number of their generators). A (non-exhaustive) list of such monoids follows; we do not reproduce the definitions but provide two references for each series: the first gives the source where the series was introduced, and the second refers to the paper that proved the equational equivalence between the $m$-th monoid in the series and $C_m$ (or $C_{m+1}$):
\begin{itemize}\itemsep-1pt
  \item the monoid of all reflexive binary relations on an $m$-element sets \cite{St80}, \cite{Vo04};
  \item the monoid of all unitriangular Boolean $m\times m$-matrices \cite{St80}, \cite{Vo04};
  \item the Kiselman monoid with $m$ generators \cite{KuMa09}, \cite{AVZ15};
  \item the double Catalan monoid with $m$ generators \cite{MaSt12}, \cite{JoFe19};
  \item the gossip monoid with $m$ generators \cite{BDF15}, \cite{JoFe19};
  \item the stylic monoid with $m$ generators \cite{AbRe22}, \cite{Vo22}.
\end{itemize}
Once the equational equivalence is established, Proposition~\ref{prop:FBP}(b) gives the absence of a finite identity basis for monoids with $m\ge 5$ in the first two items of the list and $m\ge 4$ in the other items. Corollary~\ref{cor:C5} yields the same result but in an easier way since it requires only `one half' of the equational equivalence: it suffices to show that $\Id C_5$ contains the equational theory of the corresponding monoid. In some cases (say, for Kiselman, double Catalan, or stylic monoids), this is much simpler to show than the opposite inclusion.

\paragraph{3.} Goldberg~\cite{Go06a,Go06b,Go07} systematically studied the FBP for monoids of partial order preserving and/or extensive transformations. Along with the series $\{E_m\}_{m\ge1}$ and $\{IC_m\}_{m\ge1}$ that we discussed in Section~\ref{sec:intro2}, he examined the following transformation monoids on $[m]$:
\begin{itemize}\itemsep-1pt
  \item $PE_m$, the monoid of all partial extensive transformations;
  \item $PC_m$, the monoid of all partial extensive order preserving transformations;
  \item $IE_m$, the monoid of all partial extensive injections.
 \end{itemize}
They all were shown to be \nfb\ whenever $m\ge 4$. These results readily follow from Theorem~\ref{thm:C5} since for any $m\ge 4$, the $i$-Catalan monoid $IC_4$ is a submonoid in both $PC_m$ and $IE_m$, which in turn are submonoids in $PE_m$. Goldberg wrote \cite[p.102]{Go07}, ``Observe that the situation when in a sequence of finite transformation monoids (naturally indexed by the size of the base set) all monoids except a few ones at the beginning of the sequence are nonfinitely based is quite common. \dots\  It is very tempting to find out some general reason that forces `large enough' transformation monoids to be nonfinitely based.'' Our Theorem~\ref{thm:C5} reveals that a `general reason' sought by Goldberg is the presence of the $i$-Catalan monoid $IC_4$ in the varieties generated by transformation monoids he considered.

\subsection{Catalan monoids of acyclic graphs and stratifications of $\mathbf{R}$}
\label{subsec:strata}

Let $\Gamma=(V,E)$ with $E\subseteq V\times V$ be a directed graph (digraph); we refer to the elements of the sets $V$ and $E$ as the \emph{vertices} and, respectively, the \emph{edges} of $\Gamma$. Edges of the form $(v,v)$ are called \emph{loops}; as loops are useless for the objects that we are going to introduce, we assume that $\Gamma$ has no loops. For each edge $e=(p,q)\in E$, define the \emph{elementary transformation} $\tau_e$ of the set $V$ as the map that fixes all vertices $v\in V$ except $p$ that is sent to $q$:
\[
v\tau_e=\begin{cases}
q&\text{if \ $v=p$},\\
v&\text{if \ $v\ne p$}.
\end{cases}
\]
Solomon~\cite{So94} defined the \emph{Catalan monoid} of the digraph $\Gamma$, denoted by $C(\Gamma)$, as the submonoid generated by the set $\{\tau_e\mid e\in E\}$ in the monoid of all transformations of the set $V$. The Catalan monoids $C_m$ defined in Section~\ref{sec:intro2} are special instances of this construction: namely, the monoid $C_m$ can be identified with the monoid $C(P_m)$ where $P_m$ stands for the directed simple path with $m$ vertices:
\begin{equation}
\label{eq:path}
P_m: \ \ \stackrel{1}{\bullet}\ \longrightarrow \ \stackrel{2}{\bullet}\ \longrightarrow\dots\longrightarrow \stackrel{m{-}1}{\bullet}\longrightarrow \ \stackrel{m}{\bullet}.
\end{equation}
Therefore, if a digraph $\Gamma$ contains a directed simple path with at least five vertices, that is, a sequence $m\ge 5$ distinct vertices $v_1,v_2\dots,v_m$ such that $(v_i,v_{i+1})\in E$ for all $i=1,\dots,m-1$, then the Catalan monoid $C(\Gamma)$ has a submonoid isomorphic to $C_5$.

A digraph $\Gamma=(V,E)$ is said to be \emph{acyclic} if it has no directed cycles, that is, no vertex sequences $v_0,v_1\dots,v_{n-1}$ with $(v_i,v_{i+1\!\!\pmod n})\in E$ for all $i=0,1,\dots,n-1$. The Catalan monoid of a finite acyclic digraph is $\mR$-trivial \cite[Corollary 2.2]{So94}. Combining this fact and Corollary~\ref{cor:C5}, we get the following:
\begin{proposition}
\label{prop:digraph}
The Catalan monoid of every finite acyclic digraph containing a directed path with at least five vertices is \nfb.
\end{proposition}

In~\cite{So94}, Catalan monoids of digraphs arose as a tool for constructing stratifications of the class of all finite $\mR$-trivial monoids. A class $\mathbf{P}$ of finite monoids is called a \emph{pseudovariety} if $\mathbf{P}$ is closed under forming finite direct products and taking divisors of monoids from $\mathbf{P}$. The classes $\mathbf{R}$, $\mathbf{J}$, and $\mathbf{J}\cap\mathbf{Ecom}$ introduced in Section~\ref{sec:intro2} all constitute pseudovarieties; also, for any semigroup variety, its \emph{trace}, that is, the class of all its finite monoids, is a pseudovariety. A \emph{stratification} of a pseudovariety $\mathbf{P}$ is an infinite, strictly increasing sequence of traces (called \emph{strata})
\[
\mathbf{P}_1\subset\mathbf{P}_2\subset\cdots\subset\mathbf{P}_n\subset\cdots
\]
such that $\mathbf{P}=\bigcup\limits_{n=1}^{\infty}\mathbf{P}_n$. The idea is that the strata $\mathbf{P}_1,\mathbf{P}_2,\dots,\mathbf{P}_n,\dots$ can be easier to deal with so that studying them layer by layer can turn out to be a reasonable way to gradually gain fine-grained information about $\mathbf{P}$. This approach has been applied to some other pseudovarieties of importance; for example, Simon used it to study the pseudovariety $\mathbf{J}$ in his thesis~\cite{Si72} (where the term `hierarchy' was used for what is called `stratification' here).

Departing from Eilenberg's characterization of the class of languages corresponding to the pseudovariety $\mathbf{R}$ of all finite $\mR$-trivial monoids (see \cite[Theorem IV.3.3]{Pin86}),  Solomon~\cite[Section~2]{So94} introduced a stratification of $\mathbf{R}$ that he called \emph{Catalan}. The $n$-th stratum of the Catalan stratification is the trace of the variety $\var C(\Gamma_n)$ where $\Gamma_n$ is the acyclic digraph shown in Fig.~\ref{fig:tree}.
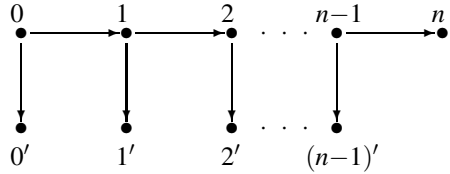
\begin{figure}[ht]
\unitlength=1.4mm
\begin{picture}(60,14)(-12,9)
\multiput(13,20)(10,0){5}{\circle*{1}}
\multiput(13.9,20)(10,0){2}{\vector(1,0){8.4}}
\multiput(13,19)(10,0){4}{\vector(0,-1){7.4}}
\put(43.9,20){\vector(1,0){8.4}}
\multiput(13,11)(10,0){4}{\circle*{1}}
\multiput(36,11)(2,0){3}{\circle*{0.2}}
\multiput(36,20)(2,0){3}{\circle*{0.2}}
\small
\put(12,21){0}
\put(22,21){1}
\put(32,21){2}
\put(41,21){$n{-}1$}
\put(52,21){$n$}
\put(40,7.5){$(n{-}1)'$}
\put(32,7.5){$2'$}
\put(22,7.5){$1'$}
\put(12,7.5){$0'$}
\end{picture}
\caption{The digraph $\Gamma_n$}\label{fig:tree}
\end{figure}

Goldberg \cite[Theorem 3.1]{Go06b} proved that the monoid $C(\Gamma_n)$ is \nfb\ for each $n\ge4$. Of course, this is a special instance of Proposition~\ref{prop:digraph} since the digraph $\Gamma_n$ is acyclic and contains a directed path with $n+1$ vertices. Theorem~\ref{thm:C5} implies a similar fact for any stratification of the pseudovariety $\mathbf{R}$ whose strata are traces of varieties generated by a finite monoid, and moreover, the result requires no a priori information on the structure of the generating monoids. Indeed, consider such a stratification $\{\mathbf{R}_n\}_{n\ge 1}$ of $\mathbf{R}$. Since $\mathbf{R}_1\subset\mathbf{R}_2\subset\cdots\subset\mathbf{R}_n\subset\cdots$ and  $\mathbf{R}=\bigcup\limits_{n=1}^{\infty}\mathbf{R}_n$, the $i$-Catalan monoid $IC_4$, which is $\mJ$-trivial, and hence, $\mR$-triv\-i\-al, must belong to each $\mathbf{R}_n$ with $n$ greater than certain $n_0$. Let $M_n$ be a finite monoid such that $\mathbf{R}_n$ is the trace of the variety $\var M_n$. Then $M_n\in\mathbf{R}_n$ by the definition of a trace whence $M_n$ is $\mR$-trivial. Therefore, for each $n>n_0$, the monoid $M_n$ is \nfb\ by Theorem~\ref{thm:C5}.

Clearly, the same argument applies to any stratification of $\mathbf{J}$ or $\mathbf{J}\cap\mathbf{Ecom}$ whose strata are traces of varieties generated by a finite monoid.

\subsection{Free tree monoids}
\label{subsec:tree}

Ayyer et al.\ \cite{ASST15} have developed a general theory of Markov chains realizable as random walks on $\mR$-trivial monoids, thus providing an elegant and uniform treatment of many classical examples and their generalizations. An essential role in the considerations in \cite{ASST15} is played by a novel series of finite $\mR$-trivial monoids, so-called free tree monoids. The FBP for these monoids does not seem to have been studied so far, but here we demonstrate that it is quite amenable to our approach.

From now on we assume the reader's acquaintance with presenting of monoids in terms of generators and relations; see~\cite[Section 1.12]{ClPr61} or \cite[Section 1.6]{Ho95}. We will frequently use the following fact which is a specialization of Dyck's Theorem (see, e.g., \cite[Theorem III.8.3]{Cohn}) to the case of monoids.
\begin{lemma}\label{lem:dyck}
Let $M$ and $N$ be monoids such that $M$ is generated by a set $A$ subject to relations $R$ and $N$ is generated by $\varphi(A)$ for some map $\varphi\colon A\to N$. If all relations obtained from $R$ by substituting each $a\in A$ with $\varphi(a)$ hold in $N$, then the map $\varphi$ extends to a homomorphism of $M$ onto $N$.
\end{lemma}

A convenient presentation for the Catalan monoid $C_m$ was found by Solomon \cite[Section 9]{So96}; see also \cite{GaMa11} for a short argument. Namely, $C_m$ can be identified with the monoid generated by $a_1,a_2,\dots,a_{m-1}$ subject to the relations
\begin{align}
&a_i^2=a_i                  &&\text{for each } i=1,\dots,m-1;\label{eq:idempotent2}\\
&a_{i}a_{k}=a_{k}a_{i}      &&\text{if }|i-k|\ge 2,\ i,k=1,\dots,m-1;\label{eq:catalan1}\\
&a_ia_{i+1}a_i=a_{i+1}a_ia_{i+1}=a_{i+1}a_i &&\text{for each } i=1,\dots,m-2.\label{eq:catalan2}
\end{align}
In the incarnation of $C_m$ as $C(P_m)$, the Catalan monoid of the directed path $P_m$ (see~\eqref{eq:path}), the role of the generators $a_i$, $i=1,\dots,m-1$, is played by the elementary transformations $\tau_{(i,i+1)}$.

In \cite{ASST15} the monoid generated by $a_1,a_2,\dots,a_n$ subject to the relations
\begin{align}
&a_i^2=a_i                  &&\text{for each } i=1,\dots,n;\label{eq:idempotent3}\\
&a_{k}a_{i}a_{k}=a_{k}a_{i} &&\text{if }1\le i<k\le n,\label{eq:treerel}
\end{align}
is named the \emph{free tree monoid}; we denote it by $FT_n$. (The name comes from the fact that the elements of $FT_n$ are in a 1-1 correspondence with certain trees.) If $t(n)=|FT_n|$ for $n\ge1$ and $t(0) = 1$, then the sequence $\{t(n)\}_{n\ge0}$ satisfies the recursion $t(n) = t(n-1)(t(n-1)+1)$; see \cite[Section 5.1]{ASST15} for details. Hence the first six free tree monoids have cardinalities
\[
2,\ 6,\ 42,\ 1806,\ 3263442,\ 10650056950806.
\]

The next straightforward observation establishes a connection between the free tree monoids and the Catalan monoids:
\begin{lemma}
\label{lem:tree-mapsonto-cat}
For each $n=1,2,\dotsc$, the Catalan monoid $C_{n+1}$ is a homomorphic image of the free tree monoid $FT_n$.
\end{lemma}

\begin{proof}
The monoids $FT_n$ and $C_{n+1}$ are both generated by the set $\{a_1,a_2,\dots,a_n\}$. By Lemma~\ref{lem:dyck}, to show that the identity map on this set extends to a homomorphism of $FT_n$ onto $C_{n+1}$, it suffices to verify that the generators $a_1,a_2,\dots,a_n$ of $C_{n+1}$ satisfy the relations \eqref{eq:idempotent3} and \eqref{eq:treerel}. It is clear for  \eqref{eq:idempotent3} in view of  \eqref{eq:idempotent2}. To verify \eqref{eq:treerel}, let $1\le i<k\le n$. If $i+1<k$, then $k-i\ge2$, and $a_{k}a_{i}a_{k}\stackrel{\eqref{eq:catalan1}}{=}a_{k}a_{k}a_{i}\stackrel{\eqref{eq:idempotent2}}{=}a_{k}a_{i}$. If $i+1=k$, then $a_{k}a_{i}a_{k}=a_{i+1}a_ia_{i+1}\stackrel{\eqref{eq:catalan2}}{=}a_{i+1}a_i=a_{k}a_{i}$.
\end{proof}

From all properties of $FT_n$ established in \cite[Section 5.1]{ASST15}, we need only the following which is a part of Corollary 5.2 in \cite{ASST15}:
\begin{lemma}
\label{lem:tree-rtriv}
The free tree monoid $FT_n$ is $\mR$-trivial.
\end{lemma}

We are ready to solve the FBP for almost all free tree monoids.
\begin{proposition}
\label{prop:tree}
For any $n\ge 4$, the free tree monoid $FT_n$ is \nfb.
\end{proposition}

\begin{proof}
By Lemma~\ref{lem:tree-rtriv} the monoid $FT_n$ is $\mR$-trivial, and Lemma~\ref{lem:tree-mapsonto-cat} implies that the Catalan monoid $C_5$ belongs to the variety $\var FT_4$. Clearly, for every $n\ge 4$, the submonoid of $FT_n$ generated by $a_1,a_2,a_3,a_4$ is isomorphic to $FT_4$ whence $C_5$ lies in the variety $\var FT_n$. Therefore, Corollary~\ref{cor:C5} applies.
\end{proof}

Amongst the free tree monoids not covered by Proposition~\ref{prop:tree}, the 2-element monoid $FT_1$ is obviously \fb. The fact that the 6-element monoid $FT_2$ also is \fb\ follows from \cite{LeLi11} where it is shown that only two 6-element monoids are \nfb. The two exceptional monoids are not $\mR$-trivial whence neither is isomorphic to $FT_2$. These observations and Proposition~\ref{prop:tree} reduce the FBP for the free tree monoids to the question of whether or not the 42-element monoid $FT_3$ is \fb. This question is still open.

In fact, Ayyer et al.\ \cite{ASST15} have introduced and studied many more finite $\mR$-trivial monoids useful for the theory of Markov chains. Our technique applies to the FBP for a good deal of such monoids, but we have restricted ourselves to just one typical application to avoid introducing plenty of extra notions.

\subsection{0-Hecke monoids}
\label{subsec:hecke}

The concept of a 0-Hecke monoid comes from the theory of Coxeter groups, classical objects binding algebra, geometry and combinatorics; see \cite{BjBr05} for an accessible introduction into that rich area. Recall the definition of Coxeter groups in terms of generators and relations.

Let $\mathbb{N}^\infty$ stand for the set all positive integers with the extra symbol $+\infty$ added. A symmetric matrix $CD=(m_{ij})_{n\times n}$ with entries in $\mathbb{N}^\infty$ is called a \emph{Coxeter matrix} if $m_{ii}=1$ for all $i$ and $m_{ij}\ge 2$ for all $i\ne j$. We depict such a matrix as the graph (called \emph{Coxeter diagram}) with vertices $1,2,\dots,n$ that has the edge ${i}{\begin{picture}(14,7)\linethickness{1pt}\put(3,3){\line(1,0){10}}\end{picture}}{j}$ if and only if $m_{ij}\ge3$; in addition, if $m_{ij}>3$, then the edge is labeled $m_{ij}$. For instance, the Coxeter matrix $\left(\begin{smallmatrix}
 1 &  4 &  2 \\
 4 &  1 &  3 \\
 2 &  3 &  1
\end{smallmatrix}\right)$ is depicted by the Coxeter diagram $B_3\colon \bullet{\begin{picture}(14,7)\linethickness{1pt}\put(3,3){\line(1,0){10}}\put(5.5,-4.5){\footnotesize4}\end{picture}}\bullet{\begin{picture}(14,7)\linethickness{1pt}\put(2.5,3){\line(1,0){10}}\end{picture}}\bullet$. (As it is common, we omit the vertex names whenever they are clear.)

If $CD=(m_{ij})_{n\times n}$ is a Coxeter matrix, then the \emph{Coxeter group} $W(CD)$ is the group generated by $s_1,s_2,\dots,s_n$ subject to the relations
\begin{equation}\label{eq:coxeter}
(s_is_j)^{m_{ij}}=1\ \text{ for all }\ i,j=1,2,\dots,n\ \text{ such that }\ m_{ij}\ne+\infty.
\end{equation}
Since $m_{ii}=1$, the relations \eqref{eq:coxeter} for $i=j$ mean $s_i^2=1$, that is, each generator $s_i$ is an involution. Using this, one can rewrite the relations \eqref{eq:coxeter} for $i\ne j$ as
\begin{equation}\label{eq:braid}
\underbrace{s_is_j\cdots}_{\text{$m_{ij}$ factors}}=\underbrace{s_js_i\cdots}_{\text{$m_{ij}$ factors}}.
\end{equation}
Continuing our example, the Coxeter group $W(B_3)$ is generated by $s_1,s_2,s_3$ subject to the following six relations:
\[
s_1^2=s_2^2=s_3^2=1,\quad s_1s_2s_1s_2=s_2s_1s_2s_1,\quad s_1s_3=s_3s_1,\quad s_2s_3s_2=s_3s_2s_3.
\]
It has order 48 and is realizable as the group of all symmetries of the usual cube.

The 0-\emph{Hecke monoid} of the group $W(CD)$ is the monoid $H_0(CD)$ generated by $s_1,s_2,\dots,s_n$ subject to the relations \eqref{eq:braid} for all $i\ne j$ such that $m_{ij}\ne+\infty$ and
\begin{equation}\label{eq:idempotent}
s_i^2=s_i \ \text{ for all }\ i=1,2,\dots,n.
\end{equation}
Thus, one passes from $W(CD)$ to $H_0(CD)$ by merely converting each involution $s_i$ into an idempotent with the same name.

Even though the 0-Hecke monoid of a Coxeter group radically differs from the group as an algebraic object, the monoid and the group share many combinatorial features. The reason for this is that the elements of $W(CD)$ and $H_0(CD)$ can be shown to be representable as the same reduced words in the generators $s_1,s_2,\dots,s_n$, albeit with different multiplication rules (see \cite[Theorem 1]{Ts90} where 0-Hecke monoids appear as \emph{Coxeter monoids}). In particular, the Coxeter group $W(CD)$ is finite if and only if so is its 0-Hecke monoid $H_0(CD)$, and moreover, $|H_0(CD)|=|W(CD)|$.

The following property of finite 0-Hecke monoids is explicitly mentioned, e.g., in \cite{DHST11}, see Sections 2.3 and 2.4 of that paper.
\begin{lemma}
\label{lem:hecke}
Each finite $0$-Hecke monoid is $\mJ$-trivial.
\end{lemma}

Let $A_n$ stand for the unlabeled simple path with $n$ vertices:
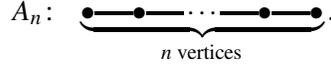
\[
A_n\colon \ \ \underbrace{{\bullet}\rule[2pt]{14pt}{1pt}{\bullet}\rule[2pt]{14pt}{1pt}\cdots\rule[2pt]{14pt}{1pt}{\bullet}\rule[2pt]{14pt}{1pt}{\bullet}}_{\text{$n$ vertices}}.
\]
The relations \eqref{eq:coxeter} defined by the Coxeter diagram $A_n$ are nothing but Moore's classical relations \cite[Theorem A]{Moore97} for the symmetric group $\mathbb{S}_{n+1}$ so that the Coxeter group $W(A_n)$ is isomorphic to $\mathbb{S}_{n+1}$. It was observed in the literature that the 0-Hecke monoid $H_0(A_n)$ projects onto the Catalan monoid $C_{n+1}$; see, e.g., \cite[Theorem 1(viii)]{GaMa11} or \cite[Section 5]{HiTh09}\footnote{A transformation $\alpha\colon[m]\to[m]$ is called \emph{decreasing} or \emph{parking} if $i\alpha\le i$ for all $i\in[m]$. In \cite{GaMa11} and \cite{HiTh09}, their authors work with monoids of order preserving decreasing transformations but this makes no difference since the monoid of all such transformations on $[n+1]$ is isomorphic to $C_{n+1}$.}. The same argument yields a more general fact:
\begin{lemma}
\label{lem:subdiagram}
Suppose that a Coxeter diagram $CD$ has a simple path with $n$ vertices \textup(whose edges may bear labels\textup). Then the Catalan monoid $C_{n+1}$ is a divisor of the $0$-Hecke monoid $H_0(CD)$.
\end{lemma}

\begin{proof}
Renumbering the vertices of the diagram $CD$ if necessary, we may assume that the path from the premise of the lemma is formed by the vertices $1,2,\dots,n$. Consider the subgraph $CD_n$ induced by $CD$ on these $n$ vertices. Then $CD_n$ also is a Coxeter diagram, and the $0$-Hecke monoid $H_0(CD_n)$ is a submonoid in $H_0(CD)$.

Consider the bijection $s_i\mapsto a_i$ between the generators of $H_0(CD_n)$ and $C_{n+1}$. By Lemma~\ref{lem:dyck}, it extends to a homomorphism of $H_0(CD_n)$ onto $C_{n+1}$ if $a_1,a_2,\dots,a_n$ satisfy the relations \eqref{eq:idempotent} and \eqref{eq:braid} that are imposed on $s_1,s_2,\dots,s_n$ in the definition of $H_0(CD_n)$. It is clear for \eqref{eq:idempotent} in view of \eqref{eq:idempotent2}. Thus, it remains to verify that
\begin{equation}\label{eq:braid for a}
\underbrace{a_ia_j\cdots}_{\text{$m_{ij}$ factors}}=\underbrace{a_ja_i\cdots}_{\text{$m_{ij}$ factors}} \text{ for all }\ i\ne j\ \text{ such that }\ m_{ij}\ne+\infty.
\end{equation}
If $|i-j|\ge 2$, then $a_i$ and $a_j$ commute by \eqref{eq:catalan1}. Using this and \eqref{eq:idempotent2}, we obtain that both sides of \eqref{eq:braid for a} are equal to $a_ia_j$. If $|i-j|=1$, then $i$ and $j$ are adjacent in the path formed by $1,2,\dots,n$. By the definition of a Coxeter diagram, it means that $m_{ij}\ge 3$. We prove  \eqref{eq:braid for a}, inducting on $m_{ij}$. If $m_{ij}=3$, then \eqref{eq:braid for a} reduces to $a_ja_ia_j=a_ia_ja_i$, which equality holds in $C_{n+1}$ because of \eqref{eq:catalan1}. If $m_{ij}>3$, then using $a_ja_ia_j=a_ia_ja_i$ and \eqref{eq:idempotent2}, we obtain
\[
\underbrace{a_ia_ja_ia_j\cdots}_{\text{$m_{ij}$ factors}}=\underbrace{a_ja_ia_ja_j\cdots}_{\text{$m_{ij}$ factors}}=\underbrace{a_ja_ia_j\cdots}_{\text{$m_{ij}-1$ factors}},
\]
and similarly,
\[
\underbrace{a_ja_ia_ja_i\cdots}_{\text{$m_{ij}$ factors}}=\underbrace{a_ia_ja_ia_i\cdots}_{\text{$m_{ij}$ factors}}=\underbrace{a_ia_ja_i\cdots}_{\text{$m_{ij}-1$ factors}}.
\]
Now the induction assumption applies.
\end{proof}

\begin{proposition}
\label{prop:0hecke}
A finite $0$-Hecke monoid is \nfb\ whenever a connected component of its Coxeter diagram has at least four vertices and is not $D_4$ \textup(see Fig.~\ref{fig:d4}\textup).
\begin{figure}[ht]
\unitlength=1mm
\linethickness{1pt}
\begin{picture}(50,5)(-40,0)
\multiput(13,0)(10,0){3}{\circle*{1.5}}
\multiput(14,0)(10,0){2}{\line(1,0){8}}
\put(23,1){\line(0,1){8}}
\put(23,10){\circle*{1.5}}
\end{picture}
\caption{The Coxeter diagram $D_4$}\label{fig:d4}
\end{figure}
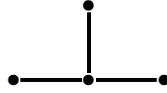
\end{proposition}

\begin{proof}
Connected Coxeter diagrams giving rise to finite Coxeter groups (and hence, to finite 0-Hecke monoids) were classified by Coxeter \cite{Co35}; the diagrams are listed, e.g., in \cite[Appendix A1, Table I]{BjBr05}. Inspecting the list readily shows that $D_4$ is the only connected Coxeter diagram with at least four vertices that has no simple path with four vertices. Now Lemma~\ref{lem:subdiagram} implies that if a finite 0-Hecke monoid $H_0$ satisfies the premise of the proposition, then the Catalan monoid $C_5$ is a divisor of $H_0$ whence $C_5$ lies in $\var H_0$. By Lemma~\ref{lem:hecke} $H_0$ is a $\mJ$-trivial monoid. We are therefore in a position to invoke Corollary~\ref{cor:C5}, which implies the claim.
\end{proof}

Finite 0-Hecke monoids with connected Coxeter diagrams whose FBP is covered by neither Proposition~\ref{prop:0hecke} nor the classification of \fb\ monoids with $\le6$ elements from \cite{LeLi11} restrict to the four monoids whose diagrams are $A_3$, $B_3$, $H_3\colon\bullet{\begin{picture}(13,7)\linethickness{1pt}\put(2,3){\line(1,0){10}}\put(5.5,-4.5){\footnotesize5}\end{picture}}\bullet{\begin{picture}(13,7)\linethickness{1pt}\put(2,3){\line(1,0){9}}\end{picture}}\bullet$, and $D_4$, plus the one-parameter series of $2n$-element monoids whose diagrams are $I_n\colon \bullet{\begin{picture}(13,7)\linethickness{1pt}\put(2,3){\line(1,0){10}}\put(5.5,-4.5){\footnotesize$n$}\end{picture}}\bullet$, $n=4,5,\dotsc$; these are the 0-Hecke monoids of the corresponding dihedral groups.

For sake of completeness, we mention that the monoids $H(I_n)$ with $n\ge 6$ are \nfb. This fact does not seem to have been registered in the literature, but it is an easy consequence of the first-named author's results on Lee monoids~\cite{Sa18}. The Lee monoids $L_n^1$, $n=3,4,\dotsc$, are given by the following monoid presentation:
\begin{equation}\label{eq:leemonoid}
L_n^1=\langle e,f\mid e^2=e,\ f^2=f,\ \underbrace{efe\cdots}_{\text{$n$ factors}}=\underbrace{fefe\cdots}_{\text{$n+1$ factors}}=\underbrace{efef\cdots}_{\text{$n+1$ factors}}\rangle.
\end{equation}
Comparing the relations in \eqref{eq:leemonoid} with the relations
\[
s_1^2=s_1,\ s_2^2=s_2,\ \underbrace{s_1s_2\cdots}_{\text{$n+1$ factors}}=\underbrace{s_2s_1\cdots}_{\text{$n+1$ factors}}
\]
of the 0-Hecke monoid $H(I_{n+1})$, one readily obtains from Lemma~\ref{lem:dyck} that the bijection $\begin{cases}s_1\mapsto e\\ s_2\mapsto f\end{cases}$ extends to a homomorphism of $H(I_{n+1})$ onto $L_n^1$. On the other hand, in $H(I_n)$ one has the relation
\begin{equation}\label{eq:In}
\underbrace{s_1s_2\cdots}_{\text{$n$ factors}}=\underbrace{s_2s_1\cdots}_{\text{$n$ factors}}.
\end{equation}
Multiplying \eqref{eq:In} through on the left by $s_1$ and by $s_2$, one gets $\underbrace{s_1s_2\cdots}_{\text{$n$ factors}}=\underbrace{s_1s_2s_1\cdots}_{\text{$n+1$ factors}}$ and, respectively, $\underbrace{s_2s_1s_2\cdots}_{\text{$n+1$ factors}}=\underbrace{s_2s_1\cdots}_{\text{$n$ factors}}$. Combining these two equalities with \eqref{eq:In}, we see that the generators $s_1$ and $s_2$ of $H(I_n)$ fulfill
\[
\underbrace{s_1s_2\cdots}_{\text{$n$ factors}}=\underbrace{s_2s_1s_2\cdots}_{\text{$n+1$ factors}}=\underbrace{s_1s_2s_1\cdots}_{\text{$n+1$ factors}}.
\]
In view of \eqref{eq:leemonoid}, Lemma~\ref{lem:dyck} implies that the bijection $\begin{cases}e\mapsto s_1\\ f\mapsto s_2\end{cases}$ extends to a homomorphism of $L_n^1$ onto $H(I_n)$. We conclude that $\var L_{n-1}^1\subseteq\var H(I_n) \subseteq\var L_n^1$ for all $n=4,5,\dotsc$. By \cite[Corollary 2.5]{Sa18}, every monoid $M$ such that
\[
\var L_5^1 \subseteq \var M \subseteq \var L_n^1 \ \text{ for some $n$}
\]
is \nfb. Hence, for each $n\ge 6$, the 0-Hecke monoid $H(I_n)$ is \nfb.

Summarizing, we see the FBP remains open for only six finite 0-Hecke monoids with connected Coxeter diagrams; the corresponding diagrams are $I_4$, $I_5$, $A_3$, $B_3$, $H_3$ and $D_4$. The sizes of these six monoids are 8, 10, 24, 48, 120, and 192.

\subsection{Monoids of unitary subsets}
\label{subsec:unitary}

For an arbitrary monoid $M$, one can multiply its subsets element-wise: for any $A,B\subseteq M$, put $A\cdot B=\{ab\mid a\in A,\ b\in B\}$. It is known and easy to verify that this multiplication is associative and has the singleton $\{1\}$ as the identity element.  Thus, the powerset $\mathcal{P}(M)$ is a monoid on its own. Restricted to finite monoids, this construction has many applications in the algebraic theory of regular languages; see \cite[Chapter 11]{Almeida:95} and references therein.

We call a subset $A$ of a monoid $M$ \emph{unitary} if $1\in A$. Obviously, the set $\mathcal{P}_1(M)$ of all unitary subsets of $M$ forms a submonoid of the monoid $\mathcal{P}(M)$. Finite monoids of the form $\mathcal{P}_1(M)$ also have language-theoretic applications as discussed by Margolis and Pin~\cite{MaPi84}. They also made the following observation:
\begin{lemma}[\!{\mdseries\cite[Proposition 3.1]{MaPi84}}]
\label{lem:unitary}
If $M$ is a finite monoid, then the monoid $\mathcal{P}_1(M)$ of its unitary subsets is $\mJ$-trivial.
\end{lemma}
Thus, we have another natural family of finite $\mJ$-trivial monoids, and it seems that the FBP for this class has remained completely unexplored so far. A systematic study of monoids of unitary subsets from the viewpoint of the FBP goes beyond the scope of this paper. Here we restrict ourselves to two statements, demonstrating that Theorem~\ref{thm:C5} and Corollary~\ref{cor:C5} efficiently work for many such monoids.

A finite monoid $M$ is called \emph{aperiodic} if no non-singleton subsemigroup of $M$ is a group. An equivalent alternative definition is that for each $a\in M$, there is a positive integer $k$ such that $a^k=a^{k+1}$.
\begin{proposition}
\label{prop:p1aperiodic}
For any noncommutative aperiodic monoid $M$, there exists a positive integer $n_0$ such that for all $n\ge n_0$, the monoid of unitary subsets of the $n$-th direct power of $M$ is \nfb.
\end{proposition}

\begin{proof}
Let $M^{(n)}$ stand for the $n$-th direct power of $M$. Corollary 3.8 from~\cite{MaPi84} says that for any $\mJ$-trivial monoid $N$, there exists a positive integer $n_0$ such that $N$ is a divisor of the monoid $\mathcal{P}_1(M^{(n_0)})$. Applying this to the $i$-Catalan monoid $IC_4$, we see that $IC_4\in\var\mathcal{P}_1(M^{(n)})$ for all $n$ greater than or equal to a certain $n_0$. By Lemma~\ref{lem:unitary}, the monoid $\mathcal{P}_1(M^{(n)})$ is $\mJ$-trivial. Hence, the monoid $\mathcal{P}_1(M^{(n)})$ is \nfb\ by Theorem~\ref{thm:C5}.
\end{proof}

\begin{remark}
\label{rem:quantifier}
The parameter $n_0$ in Proposition~\ref{prop:p1aperiodic} depends on the monoid $M$. In fact, one can prove a stronger statement: there exists a positive integer $n_0$ such that for any noncommutative aperiodic monoid $M$, the monoid $\mathcal{P}_1(M^{(n)})$ is \nfb\ for all $n\ge n_0$. For this, one should find all \emph{minimal noncommutative aperiodic divisors}, that is, noncommutative aperiodic monoids $N_i$ minimal with the property that at least one of them occurs as a divisor of any given noncommutative aperiodic monoid $M$, and show that the number of the monoids $N_i$ is finite. It is easy to see that if $N$ is a divisor of $M$, then $\mathcal{P}_1(N^{(n)})$ is a divisor of $\mathcal{P}_1(M^{(n)})$ for all positive integers $n$. Therefore, we get the required $n_0$ by choosing $n_0=\max_i\{n_i\}$ where $n_i$ is chosen  for each minimal noncommutative aperiodic divisor $N_i$ so that the $i$-Catalan monoid $IC_4$ is a divisor of $\mathcal{P}_1(N_i^{(n_i)})$. Moreover, the precise values of the numbers $n_i$'s, and hence, of $n_0$ can be computed. These results will be published elsewhere as they require more structure theory of semigroups than was assumed in this paper.
\end{remark}

Now we turn to the FBP for monoids of unitary subsets whose `parent' monoids contain non-singleton subgroups. Here we are in a position to utilize the main result of the preceding section. This is ensured by the next observation that comes from~\cite[Theorem 1]{Ts90}; see also \cite[Proposition 1]{MaSt12} for a short argument for the Coxeter group $W(A_n)\cong\mathbb{S}_{n+1}$ which readily generalizes to any Coxeter group.
\begin{lemma}
\label{lem:hecke-untary}
If $CD$ is a Coxeter matrix and $s_1,s_2,\dots,s_n$ are the generators of the Coxeter group $W(CD)$, then the submonoid of the monoid $\mathcal{P}_1(W(CD))$ generated by the subsets $\{1,s_i\}$, $i=1,2,\dots,n$, is isomorphic to the $0$-Hecke monoid $H_0(CD)$.
\end{lemma}

\begin{proposition}
\label{prop:p1subgroups}
Suppose that a finite monoid $M$ has a Coxeter subgroup $W(CD)$ such that $CD$ has a connected component with at least four vertices and not equal to $D_4$. Then the monoid of unitary subsets of $M$ is \nfb.
\end{proposition}

\begin{proof}
Lemma~\ref{lem:hecke-untary} implies that the monoid $\mathcal{P}_1(M)$ has a submonoid isomorphic to the $0$-Hecke monoid $H_0(CD)$. The proof of Proposition~\ref{prop:0hecke} shows that the Catalan monoid $C_5$ is a divisor of $H_0(CD)$ whence $C_5$ lies in $\var\mathcal{P}_1(M)$. By Lemma~\ref{lem:unitary} $\mathcal{P}_1(M)$ is $\mJ$-trivial. Thus,  Corollary~\ref{cor:C5} applies to the monoid $\mathcal{P}_1(M)$.
\end{proof}

\section{Discussion and future work}
\label{sec:discussion}

We have shown that the 42-element monoids $IC_4$ and $C_5$ are inherently \nfb\ relative to finite $\mR$-trivial semigroups and demonstrated a number of applications of this result. What is next?

Having as a model Mark Sapir's work on `absolutely' inherently \nfb\ semigroups \cite{Sa87a,Sa87b}, one can set the goal of a characterization of all semigroups that are inherently \nfb\ relative to finite $\mR$-trivial semigroups. To understand how such a characterization might look like, recall the combinatorial characterization of finite inherently \nfb\ semigroups from \cite{Sa87a}.

Let $x_1,x_2,\dots,x_n,\dots$ be a sequence of variables. The sequence $\{Z_n\}_{n=1,2,\dots}$ of \emph{Zimin words} is defined inductively by $Z_1=x_1$, $Z_{n+1}=Z_nx_{n+1}Z_n$.
\begin{proposition}[\!{\mdseries\cite[Proposition 7]{Sa87a}}]
\label{prop:zimin}
A finite semigroup $S$ is inherently \nfb\ if and only if all Zimin words $Z_n$, $n=1,2,\dotsc$, are isoterms for $S$.
\end{proposition}

Analyzing our proof of Theorem~\ref{thm:C5}, one sees that the only property of the monoid $IC_4$ that has been used is the fact (established in Lemma~\ref{lem:sparseC5}) that all sparse words are isoterms for $IC_4$.
Thus, the proof actually yields the following result parallel to the `if' part of Proposition~\ref{prop:zimin}:
\begin{proposition}
\label{prop:sufficient}
A finite $\mR$-trivial semigroup $S$ is inherently \nfb\ relative to finite $\mR$-trivial semigroups if all sparse words are isoterms for $S$.
\end{proposition}
We do not know whether or not the condition of Proposition~\ref{prop:sufficient} is necessary. If it is, then we would get a combinatorial characterization of semigroups that are inherently \nfb\ relative to finite $\mR$-trivial semigroups in the flavour of Mark Sapir's result for the `absolute' case. To make the analogy even more apparent, one can restate Proposition~\ref{prop:sufficient}, requiring that only a suitable sequence of `typical' sparse words consists of isoterms for $S$.

As for the structural  characterization of finite inherently \nfb\ semigroups from \cite{Sa87b}, we have no possible analogue for our case in sight. Nor do we know whether or not 42 is the minimum cardinality of a semigroup that is inherently \nfb\ relative to finite $\mR$-trivial semigroups. There exist much smaller \nfb\ $\mJ$-trivial semigroups, the smallest being the 6-element semigroup $L_3$ given by the semigroup presentation
\[
\langle e,f\mid e^2=e,\ f^2=f,\ efe=(ef)^2=(fe)^2\rangle;
\]
see \cite{ZhLu11}. For some of such smaller examples (for instance, for $L_3$), we know that they are not inherently \nfb\ relative to finite $\mR$-trivial semigroups, but for many instances the question is still open. The smallest such instance is the 8-element semigroup $L_4$ defined by the the semigroup presentation
\[
\langle e,f\mid e^2=e,\ f^2=f,\ (ef)^2=(ef)^2e=(fe)^2f\rangle.
\]
The fact that $L_4$ is \nfb\ follows from~\cite[Theorem 6.2]{Lee17}.

Still, it is very tempting to conjecture that the 42-element monoids $IC_4$ and $C_5$ are the only semigroups of minimum size that are inherently \nfb\ relative to finite $\mR$-trivial semigroups (and not only because of the special role of the number 42 known from Douglas Adams's `The Hitchhiker's Guide to the Galaxy'). This would make quite a perfect analogy with the `absolute' case where there exists exactly two inherently \nfb\ semigroups of minimum size. These two are the 6-element \emph{Brandt monoid} $B_2^1$ (which is, quoting from \cite{JaZh21}, `perhaps the most ubiquitous harbinger of complex behaviour in all finite semigroups') and another 6-element monoid commonly denoted by $A_2^1$. Both $B_2^1$ and $A_2^1$ have nice monoid presentations and convenient faithful representations by zero-one $2\times 2$-matrices, but here we prefer to define them as certain monoids of order preserving transformations of a chain to align their definitions with those of $IC_4$ and $C_5$. Namely, the Brandt monoid $B_2^1$ is nothing but the monoid of all partial order preserving injections of the 2-element chain, see Fig.~\ref{fig:Brandt}.
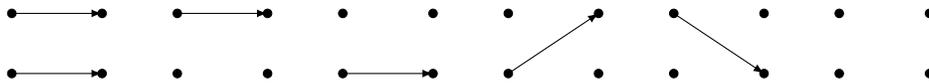
\begin{figure}[htb]
\vspace*{15pt}
\centering
\begin{tikzpicture}
[scale=0.8]
\foreach \x in {0.75,2.25,3.5,5,6.25,7.75,9,10.5,11.75,13.25,14.5,16} \foreach \y in {0.5,1.5} \filldraw (\x,\y) circle (2pt);
	\draw (0.75,0.5) edge[-latex] (2.25,0.5);
	\draw (0.75,1.5) edge[-latex] (2.25,1.5);
	\draw (3.5,1.5) edge[-latex] (5,1.5);
	\draw (6.25,0.5) edge[-latex] (7.75,0.5);
	\draw (9,0.5) edge[-latex] (10.5,1.5);
	\draw (11.75,1.5) edge[-latex] (13.25,0.5);
\end{tikzpicture}
\caption{The six partial order preserving injections forming ${B}^1_2$}\label{fig:Brandt}
\end{figure}
For comparison, the $i$-Catalan monoid $IC_4$ consists of all partial order preserving \emph{and extensive} injections of the 4-element chain. Similarly, the monoid $A_2^1$ can be defined as the monoid of all total order preserving transformations of the 3-element chain that fix the greatest element of the chain, see Fig.~\ref{fig:A2}.
\begin{figure}[htb]
\vspace*{15pt}
\centering
\begin{tikzpicture}
[scale=0.8]
\foreach \x in {0.75,2.25,3.5,5,6.25,7.75,9,10.5,11.75,13.25,14.5,16} \foreach \y in {0.5,1.5,2.5} \filldraw (\x,\y) circle (2pt);
	\draw (0.75,0.5) edge[-latex] (2.25,0.5);
	\draw (0.75,1.5) edge[-latex] (2.25,1.5);
    \draw (0.75,2.5) edge[-latex] (2.25,2.5);
	\draw (3.5,0.5) edge[-latex] (5,1.5);
	\draw (3.5,1.5) edge[-latex] (5,1.5);
    \draw (3.5,2.5) edge[-latex] (5,2.5);
	\draw (6.25,0.5) edge[-latex] (7.75,0.5);
	\draw (6.25,1.5) edge[-latex] (7.75,2.5);
    \draw (6.25,2.5) edge[-latex] (7.75,2.5);
	\draw (9,0.5) edge[-latex] (10.5,1.5);
	\draw (9,1.5) edge[-latex] (10.5,2.5);
    \draw (9,2.5) edge[-latex] (10.5,2.5);
	\draw (11.75,0.5) edge[-latex] (13.25,0.5);
	\draw (11.75,1.5) edge[-latex] (13.25,0.5);
    \draw (11.75,2.5) edge[-latex] (13.25,2.5);
    \draw (14.5,0.5) edge[-latex] (16,2.5);
	\draw (14.5,1.5) edge[-latex] (16,2.5);
    \draw (14.5,2.5) edge[-latex] (16,2.5);
\end{tikzpicture}
\caption{The six order preserving transformations forming ${A}^1_2$}\label{fig:A2}
\end{figure}
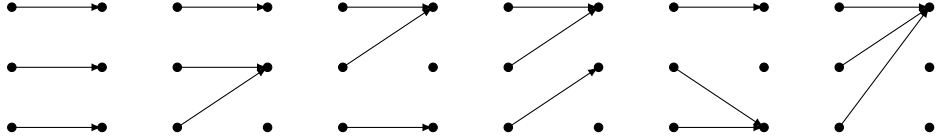
Again, for comparison, the Catalan monoid $C_5$ consists of all total order preserving \emph{and extensive} transformations of the 5-element chain, and the extensity implies that all transformations of $C_5$ fix the greatest element of the chain. We see that to get the definitions of $IC_4$ and $C_5$ from those of $B_2^1$ and $A_2^1$ respectively, one just adds two points to the base sets and includes the extensity requirement. The analogy extends even further, say, both $C_5$ and $A_2^1$ are generated by their idempotents while in both $IC_4$ and $B_2^1$ the idempotents commute. Also, the strict inclusion $\var IC_4\subset\var C_5$ that follows from Proposition~\ref{prop:inclusion} perfectly parallels the well-known inclusion $\var B^1_2\subset\var A^1_2$.

Along with trying to find semigroups that are inherently \nfb\ relative to finite $\mR$-trivial semigroups and are smaller in size than $IC_4$ and $C_5$, one can attempt to advance into the opposite direction: to enlarge the class of finite semigroups relative to which $IC_4$ and $C_5$ are inherently \nfb. Yet another look at the proof of Theorem~\ref{thm:C5} tells us that the only property of finite $\mR$-trivial semigroups that the proof exploits is Lemma~\ref{lem:noisoterm} showing that for any such semigroup, all words from a specific infinite series are non-isoterms. Besides, further analysis shows that rather than a concrete form of these words, their properties (P1) and (P2) stated after the construction \eqref{eq:construct} make the proof work. These considerations lead to the following result.
\begin{proposition}
\label{prop:newclass}
Suppose that $\mathbf{C}$ is a class of finite semigroups and for each semigroup $S\in\mathbf{C}$, there exists an infinite sequence of non-isoterms $\{\mathbf{w}_n\}$ such that
\begin{itemize}
  \item[\emph{(P1)}] for all variables $y,z$, the word $yz$ occurs in $\mathbf{w}_n$ as a factor at most once;
  \item[\emph{(P2)}] for every variable $z$, there are at least $n$ pairwise distinct variables between any two occurrences of $z$ in $\mathbf{w}_n$.
\end{itemize}
Then the $i$-Catalan monoid $IC_4$ is inherently \nfb\ relative to $\mathbf{C}$.
\end{proposition}

We give an example of a natural class of finite monoids that satisfies the conditions of Proposition~\ref{prop:newclass} and strictly contains the class $\mathbf{R}$ of all finite $\mR$-trivial monoids\footnote{This example can be easily modified to provide a similar result in the semigroup setting.}. Recall the notion dual to $\mR$-triviality: a semigroup $S$ is called $\mL$-\emph{trivial} if every principal left ideal of $S$ has a unique generator, that is, for all $a,b\in S$,
\[
Sa\cup\{a\}=Sb\cup\{b\}\to a=b.
\]
The class $\mathbf{L}$ of all finite $\mL$-trivial monoids is a pseudovariety. The class
\[
\mathbf{R}\vee\mathbf{L}=\{S\mid S \text{ is a divisor of } R\times L\ \text{ for some }\ R\in\mathbf{R},\ L\in\mathbf{L}\}
\]
is the smallest pseudovariety containing both $\mathbf{R}$ and $\mathbf{L}$. The pseudovariety $\mathbf{R}\vee\mathbf{L}$ arose in the study of formal languages in \cite{Ko85} and was further investigated in~\cite{AlAz89,KuLa12}.
\par Recall that a semigroup satisfying the identity $x^2\bumpeq x$ is called a \emph{band}. The class $\mathbf{B}$ of all finite band monoids also forms a pseudovariety, and we can consider $\mathbf{R}\vee\mathbf{L}\vee\mathbf{B}$, the smallest pseudovariety containing both $\mathbf{R}\vee\mathbf{L}$ and $\mathbf{B}$. It admits a similar description in terms of division:
\[
\mathbf{R}\vee\mathbf{L}\vee\mathbf{B}=\{S\mid S \text{ is a divisor of } R\times L\times B\ \text{ for some }\ R\in\mathbf{R},\ L\in\mathbf{L},\ B\in\mathbf{B}\}.
\]
Of course, $\mathbf{R}\vee\mathbf{L}\vee\mathbf{B}$ strictly contains $\mathbf{R}$.

\begin{proposition}
\label{prop:rjoinl}
The $i$-Catalan monoid $IC_4$ is inherently \nfb\ relative to the pseudovariety $\mathbf{R}\vee\mathbf{L}\vee\mathbf{B}$.
\end{proposition}

\begin{proof}
In view of Proposition~\ref{prop:newclass}, it suffices to exhibit, for any given monoid $S\in\mathbf{R}\vee\mathbf{L}\vee\mathbf{B}$ an infinite sequence of non-isoterms $\{\mathbf{w}_n\}$ satisfying (P1) and (P2). So, take any $S\in\mathbf{R}\vee\mathbf{L}\vee\mathbf{B}$, fix some $R\in\mathbf{R}$, $L\in\mathbf{L}$, and $B\in\mathbf{B}$ such that $S$ is a divisor of $R\times L\times B$, and let $m=\max\{|R|,|L|\}$. For each $n=2,3,\dotsc$, we set
\[
\mathbf{w}_n=\mathbf{u}_{2n}(m)\mathbf{v}_{2n}(m),
\]
where the `head' $\mathbf{u}_{2n}(m)$ and the `tail' $\mathbf{v}_{2n}(m)$ are obtained by suitable modifications of the construction \eqref{eq:construct}. Namely,
\begin{equation}
\label{eq:construct1}
\mathbf{u}_{2n}(m)=xf^{m-1}(\mathbf{y}_{2n})x f^{m-2}(\mathbf{y}_{2n}) \cdots xf(\mathbf{y}_{2n}) x\mathbf{y}_{2n},
\end{equation}
where $\mathbf{y}_{2n}=y_1 y_2 \cdots y_{2n}$, the variables $y_1,y_2,\dots,y_{2n}$ are all distinct,  $f$ is the function defined after Proposition~\ref{prop:nfbcor} (recall that $f$ adds a `fresh' variable before the first variable, between each pair of adjacent variables, and after the last variable of its argument), and the variable $x$ does not occur in the word $f^{m-1}(\mathbf{y}_{2n})$. For constructing the `tail', let $\mathbf{y}'_{2n}=y_1 y_3 \cdots y_{2n-1}\cdot y_2 y_4\cdots y_{2n}$ and define
\begin{equation}
\label{eq:construct2}
\mathbf{v}_{2n}(m)=\mathbf{y}'_{2n}x\overline{f}(\mathbf{y}'_{2n})\cdots \overline{f}^{m-2}(\mathbf{y}'_{2n})x\overline{f}^{m-1}(\mathbf{y}'_{2n}),
\end{equation}
where the function $\overline{f}$ inserts the same `fresh' variables as $f$ but in the opposite order. For instance, if $m=n=2$, then we have
\[
\mathbf{w}_2=\underbrace{xp_0y_1p_1y_2p_2y_3p_3y_4p_4xy_1y_2y_3y_4}_{\mathbf{u}_{4}(2)}\cdot\underbrace{y_1y_3y_2y_4xp_4y_1p_3y_3p_2y_2p_1y_4p_0x}_{\mathbf{v}_{4}(2)}.
\]
It readily follows from the construction of the words $\mathbf{w}_n$, $n=2,3,\dotsc$, that they satisfy  the properties (P1) and (P2). It remains to verify that all these words are non-isoterms for the monoid $S$.

As in the proof of Lemma~\ref{lem:noisoterm}, we apply Proposition~\ref{prop:iden} to the $\mR$-trivial monoid $R$ with $\mathbf{u}_{2n}(m)$ and $x$ in the roles of $\mathbf{u}$ and, respectively, $\mathbf{v}$ and get that the identity $\mathbf{u}_{2n}(m)x\bumpeq\mathbf{u}_{2n}(m)$ holds in $R$. Multiplying this identity through by $\mathbf{v}_{2n}(m)$ on the right, we deduce that $R$ satisfies the identity
\begin{equation}\label{eq:product}
\mathbf{u}_{2n}(m)x\mathbf{v}_{2n}(m)\bumpeq\mathbf{u}_{2n}(m)\mathbf{v}_{2n}(m)=\mathbf{w}_n.
\end{equation}
Using the left-right symmetry, we apply the dual of Proposition~\ref{prop:iden} to the $\mL$-trivial monoid $L$ with $\mathbf{v}_{2n}(m)$ and $x$ in the roles of $\mathbf{u}$ and, respectively, $\mathbf{v}$ and get that the identity $\mathbf{v}_{2n}(m)\bumpeq x\mathbf{v}_{2n}(m)$ holds in $L$. Multiplying this identity through by $\mathbf{u}_{2n}(m)$ on the left, we deduce that $L$ also satisfies \eqref{eq:product}.

As shown by Green and Rees \cite{GrRe52} (see also \cite[\S4.5]{Ho95}), every band satisfies all identities of the form $\mathbf{u}x\mathbf{v}\bumpeq\mathbf{u}\mathbf{v}$ where $x\in\alf(\mathbf{u})=\alf(\mathbf{v})$. By the construction, $x\in\alf(\mathbf{u}_{2n}(m))=\alf(\mathbf{v}_{2n}(m))$, whence the identity \eqref{eq:product} holds in the band $B$.

Since \eqref{eq:product} holds in $R$, $L$, and $B$, it holds in the direct product $R\times L\times B$ of these monoids. As $S$ is a divisor of this product, it also satisfies \eqref{eq:product}. Thus, each of the words $\mathbf{w}_n$, $n=2,3,\dotsc$, is a non-isoterm for $S$.
\end{proof}

Even though Proposition~\ref{prop:rjoinl} has not brought new concrete applications so far, we think that it deserves attention as it demonstrates the idea of extending the range of our approach at work and also shows the price to be paid---wider coverage requires a more cumbersome construction. The ultimate goal at which one can aim here is the class $\mathbf{W}$ of all weakly finitely based semigroups. (Recall that a finite semigroup is \emph{weakly \fb} if it is not inherently \nfb.) The class $\mathbf{W}$ is a semigroup pseudovariety, which was a surprising consequence of \cite[Theorem 1]{Sa87b}; moreover, the second-named author has found a finite axiomatization of $\mathbf{W}$ in terms of so-called pseudoidentities~\cite[Proposition 4.4]{Vo00}. As Proposition~\ref{prop:newclass} shows, to prove that the $i$-Catalan monoid $IC_4$ is inherently \nfb\ relative to the pseudovariety $\mathbf{W}$, it suffices to construct an infinite sequence of non-isoterms fulfilling (P1) and (P2) for each weakly finitely based semigroup. If one succeeds, the monoid $IC_4$ will be the first example of a semigroup that is not inherently \nfb\ (in the `absolute' sense) but is not contained in any \fb\ variety generated by a finite semigroup. The question of whether or not such an example exists is a well-known open problem; see~\cite[Problem 4.4]{Vo01}.

As a final remark, observe that Proposition~\ref{prop:nfbcor}, the key tool behind all `non-finiteness' arguments in this paper, does not restrict to finite semigroups only. Therefore, our technique can be applied to show that certain infinite semigroups are \nfb. Say, for every commutative semigroup (finite or infinite), its direct product with $IC_4$ or $C_5$ is \nfb. Applications of this kind will be presented in a subsequent paper.

\paragraph{Acknowledgements.} The authors thank Edmond W. H. Lee for a number of valuable remarks.

\appendix

\section{Cardinality of the monoid $IC_m$}

Recall that $[m]$ stands for the set of the first $m$ positive integers ordered in the usual way: $1<2<\dots<m$. Here we exhibit a bijection between the monoid $IC_m$ of all partial order preserving and extensive injections of $[m]$ and the monoid $C_{m+1}$ of all total order preserving and extensive transformations of $[m+1]$.

Given a partial injection $\alpha$ of $[m]$, define a transformation $\overline\alpha\colon[m+1]\to[m+1]$, using backward induction on $k\in[m+1]$:
\begin{itemize}
  \item $(m+1)\overline\alpha=m+1$;
  \item if $k\le m$ and $k\alpha$ is defined, then $k\overline\alpha=k\alpha$; otherwise, $k\overline\alpha=(k+1)\overline\alpha$.
\end{itemize}
One can unfold the second line in definition of $\overline\alpha$ as follows: for $k\le m$,
\begin{equation}\label{eq:appendix}
k\overline\alpha=\begin{cases}
\ell\alpha &\text{if $\ell$ is the least with $k\le\ell$ in $\dom\alpha$},\\
m+1        &\text{if there is no $\ell$ with $k\le\ell$ in $\dom\alpha$}.
\end{cases}
\end{equation}

 To show that the map $\alpha\mapsto\overline\alpha$ is one-to-one, consider for each transformation $\beta\colon[m+1]\to[m+1]$ that fixes $m+1$, its restriction $\widehat\beta$ to the set
\begin{equation}
\label{eq:domain}
\{k\mid k\beta\ne m+1 \text{ and } k\ge i \text{ for all } i \text{ such that } k\beta=i\beta\}.
\end{equation}
Since $(m+1)\beta=m+1$, the set \eqref{eq:domain}, that is, the domain of $\widehat\beta$ is contained in $[m]$. Then $\widehat\beta$ can be thought as a partial injection of $[m]$, and it is easy to see that $\alpha=\widehat{\overline\alpha}$.

\begin{lemma}
\label{lem:preserving}
If a partial injection $\alpha$ is order preserving or extensive, then so is the transformation $\overline\alpha$.
\end{lemma}

\begin{proof}
First suppose that $\alpha$ is order preserving. Take $k_1,k_2\in[m+1]$ with $k_1\le k_2$; we have to verify that $k_1\overline\alpha\le k_2\overline\alpha$. If $k_2\overline\alpha=m+1$, then the claim holds. If $k_2\overline\alpha\ne m+1$, then according to \eqref{eq:appendix}, there exists a number $\ell\in\dom\alpha$ with $k_2\le\ell$ and $k_2\overline\alpha=\ell_2\alpha$ where $\ell_2$ is the least such number. Since $k_1\le k_2\le\ell$, we also have
$k_1\overline\alpha=\ell_1\alpha$ where $\ell_1$ is the least number in $\dom\alpha$ such that $k_1\le\ell_1$. The choice of $\ell_1$ and $\ell_2$ and the inequality $k_1\le k_2$ imply $\ell_1\le\ell_2$. Since $\alpha$ is order preserving, we have $\ell_1\alpha\le \ell_2\alpha$ whence $k_1\overline\alpha=\ell_1\alpha\le\ell_2\alpha=k_2\overline\alpha$.

The case where $\alpha$ is extensive is even simpler. Take $k\in[m+1]$; we have to show that $k\le k\overline\alpha$. If $k\overline\alpha=m+1$, then the claim holds. If $k\overline\alpha\ne m+1$, then according to \eqref{eq:appendix}, $k\overline\alpha=\ell\alpha$ where $\ell$ is the least number in $\dom\alpha$ such that $k\le\ell$. Since $\alpha$ is extensive, we have $\ell\le\ell\alpha$, which gives $k\le\ell\le\ell\alpha=k\overline\alpha$.
\end{proof}

\begin{lemma}
\label{lem:inverse}
If a transformation $\beta\colon[m+1]\to[m+1]$ that fixes $m+1$ is order preserving, then so is the partial injection $\widehat\beta$, and $\overline{\widehat\beta}=\beta$. In addition, if $\beta$ is extensive, then so is $\widehat\beta$.
\end{lemma}

\begin{proof}
Since $\widehat\beta$ is a restriction of $\beta$, the claims that $\widehat\beta$ is order preserving or extensive whenever so is $\beta$ follow immediately.

Let us check the equality $\overline{\widehat\beta}=\beta$. Take any $k\in[m+1]$; we have to verify that $k\overline{\widehat\beta}=k\beta$. If $k\beta=m+1$, then $\ell\beta=m+1$ for all $\ell\ge k$ since the transformation  $\beta$ preserves order. From \eqref{eq:domain} we see that no $\ell$ with $\ell\ge k$ lies in $\dom\widehat\beta$. Hence, according to \eqref{eq:appendix}, we get $k\overline{\widehat\beta}=m+1$, that is, $k\overline{\widehat\beta}=k\beta$.

So, assume that $k\beta\ne m+1$. Since $(m+1)\beta=m+1$, this implies that $k\in[m]$. According to \eqref{eq:appendix}, $k\overline{\widehat\beta}=\ell\widehat\beta$ where $\ell$ is the least number in $\dom\widehat\beta$ with $k\le\ell$. Since $\widehat\beta$ is a restriction of $\beta$ and $\ell$ lies in $\dom\widehat\beta$, we have $\ell\widehat\beta=\ell\beta$. Let
\[
j=\max\{i \mid i\beta=k\beta\}.
\]
Then $j\ge k$ and $j\in\dom\widehat\beta$ by \eqref{eq:domain}. Since $\ell$ is the least number with these two properties, we conclude that $\ell\le j$. Thus, $k\le\ell\le j$ whence $k\beta\le\ell\beta\le j\beta$ as $\beta$ preserves order. But $j\beta=k\beta$, and therefore, $\ell\beta=k\beta$. Hence $k\overline{\widehat\beta}=\ell\beta=k\beta$.
\end{proof}

Lemmas~\ref{lem:preserving} and~\ref{lem:inverse} imply that the map $\alpha\mapsto\overline\alpha$ is a bijection of $IC_m$ onto $C_{m+1}$. In particular, $|IC_m|=|C_{m+1}|$. Fig.~\ref{fig:bijection} on the next page illustrates the bijection for $m=3$. Perhaps, it is worth explicitly stating that, except for $m=1$, this bijection is not a monoid isomorphism.

Even though it is not essential for the present paper, we mention that by Lemmas~\ref{lem:preserving} and~\ref{lem:inverse}, the map $\alpha\mapsto\overline\alpha$ also gives a bijection of the monoid of all partial order preserving injections of $[m]$ onto the monoid of all total order preserving transformations of $[m+1]$ that fix $m+1$. For $m=2$ these are the monoids $B_2^1$ and respectively $A_2^1$ discussed in Section~\ref{sec:discussion}. (In fact, the diagrams in Fig.~\ref{fig:Brandt} and~\ref{fig:A2} are vertically aligned according to the bijection $\alpha\mapsto\overline\alpha$.)

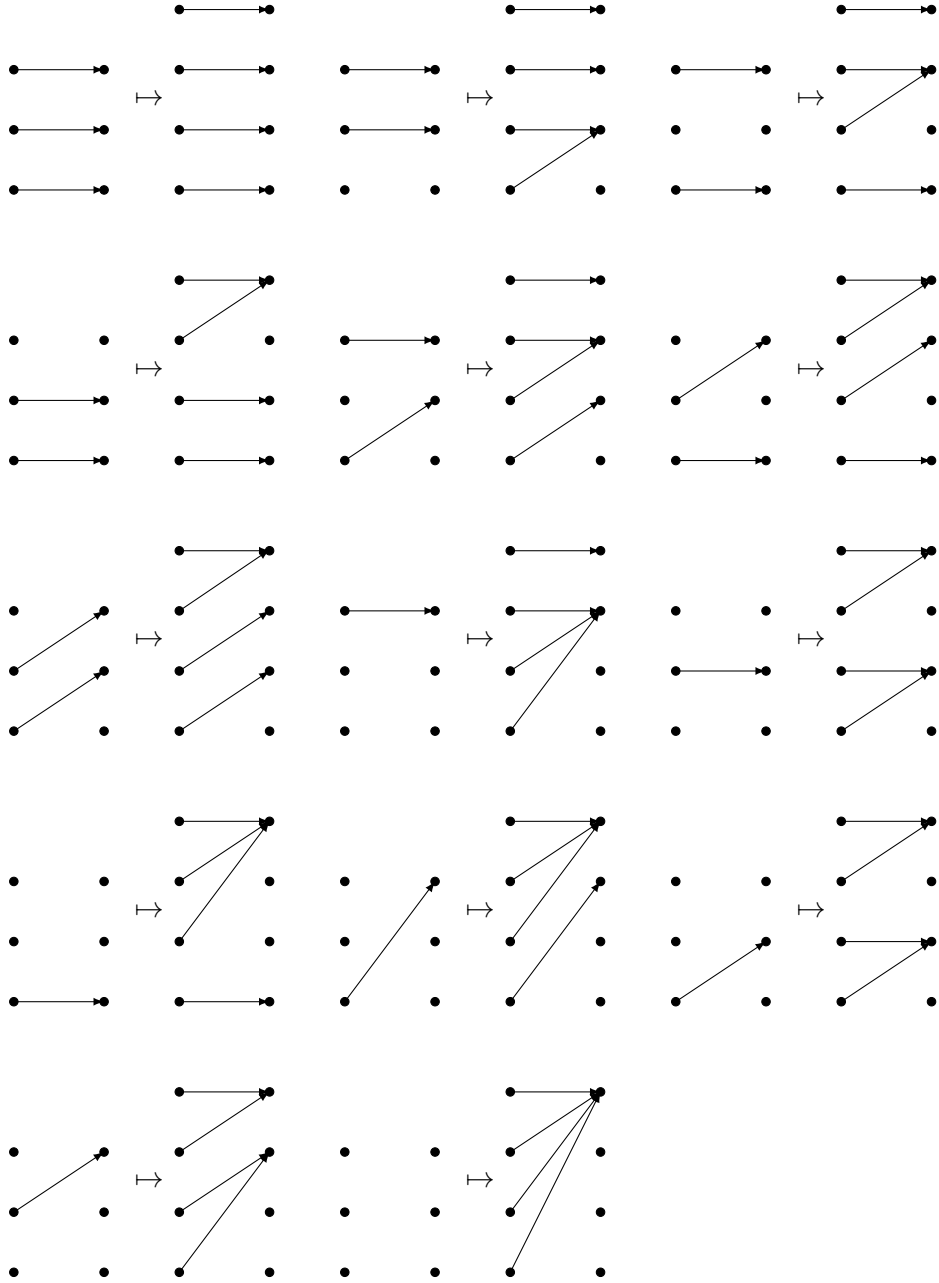
\begin{figure}[p]
\centering
\begin{tikzpicture}
[scale=0.8]
\foreach \x in {3.5,5,9,10.5,14.5,16} \foreach \y in {8,12.5,17,21.5} \filldraw (\x,\y) circle (2pt);

\foreach \x in {0.75,2.25,3.5,5,6.25,7.75,9,10.5,11.75,13.25,14.5,16} \foreach \y in {5,6,7,9.5,10.5,11.5,14,15,16,18.5,19.5,20.5} \filldraw (\x,\y) circle (2pt);

\foreach \x in {3.5,5,9,10.5} \foreach \y in {3.5} \filldraw (\x,\y) circle (2pt);

\foreach \x in {0.75,2.25,3.5,5,6.25,7.75,9,10.5} \foreach \y in {0.5,1.5,2.5} \filldraw (\x,\y) circle (2pt);

\draw (0.75,18.5) edge[-latex] (2.25,18.5);
\draw (0.75,19.5) edge[-latex] (2.25,19.5);
\draw (0.75,20.5) edge[-latex] (2.25,20.5);
\node[] at (3,20) {$\mapsto$};
\draw (3.5,18.5) edge[-latex] (5,18.5);
\draw (3.5,19.5) edge[-latex] (5,19.5);
\draw (3.5,20.5) edge[-latex] (5,20.5);
\draw (3.5,21.5) edge[-latex] (5,21.5);

\draw (6.25,19.5) edge[-latex] (7.75,19.5);
\draw (6.25,20.5) edge[-latex] (7.75,20.5);
\node[] at (8.5,20) {$\mapsto$};
\draw (9,18.5) edge[-latex] (10.5,19.5);
\draw (9,19.5) edge[-latex] (10.5,19.5);
\draw (9,20.5) edge[-latex] (10.5,20.5);
\draw (9,21.5) edge[-latex] (10.5,21.5);

\draw (11.75,18.5) edge[-latex] (13.25,18.5);
\draw (11.75,20.5) edge[-latex] (13.25,20.5);
\node[] at (14,20) {$\mapsto$};
\draw (14.5,18.5) edge[-latex] (16,18.5);
\draw (14.5,19.5) edge[-latex] (16,20.5);
\draw (14.5,20.5) edge[-latex] (16,20.5);
\draw (14.5,21.5) edge[-latex] (16,21.5);

\draw (0.75,14) edge[-latex] (2.25,14);
\draw (0.75,15) edge[-latex] (2.25,15);
\node[] at (3,15.5) {$\mapsto$};
\draw (3.5,14) edge[-latex] (5,14);
\draw (3.5,15) edge[-latex] (5,15);
\draw (3.5,16) edge[-latex] (5,17);
\draw (3.5,17) edge[-latex] (5,17);

\draw (6.25,14) edge[-latex] (7.75,15);
\draw (6.25,16) edge[-latex] (7.75,16);
\node[] at (8.5,15.5) {$\mapsto$};
\draw (9,14) edge[-latex] (10.5,15);
\draw (9,15) edge[-latex] (10.5,16);
\draw (9,16) edge[-latex] (10.5,16);
\draw (9,17) edge[-latex] (10.5,17);

\draw (11.75,14) edge[-latex] (13.25,14);
\draw (11.75,15) edge[-latex] (13.25,16);
\node[] at (14,15.5) {$\mapsto$};
\draw (14.5,14) edge[-latex] (16,14);
\draw (14.5,15) edge[-latex] (16,16);
\draw (14.5,16) edge[-latex] (16,17);
\draw (14.5,17) edge[-latex] (16,17);

\draw (0.75,9.5) edge[-latex] (2.25,10.5);
\draw (0.75,10.5) edge[-latex] (2.25,11.5);
\node[] at (3,11) {$\mapsto$};
\draw (3.5,9.5) edge[-latex] (5,10.5);
\draw (3.5,10.5) edge[-latex] (5,11.5);
\draw (3.5,11.5) edge[-latex] (5,12.517);
\draw (3.5,12.5) edge[-latex] (5,12.5);

\draw (6.25,11.5) edge[-latex] (7.75,11.5);
\node[] at (8.5,11) {$\mapsto$};
\draw (9,9.5) edge[-latex] (10.5,11.5);
\draw (9,10.5) edge[-latex] (10.5,11.5);
\draw (9,11.5) edge[-latex] (10.5,11.5);
\draw (9,12.5) edge[-latex] (10.5,12.5);

\draw (11.75,10.5) edge[-latex] (13.25,10.5);
\node[] at (14,11) {$\mapsto$};
\draw (14.5,9.5) edge[-latex] (16,10.5);
\draw (14.5,10.5) edge[-latex] (16,10.5);
\draw (14.5,11.5) edge[-latex] (16,12.5);
\draw (14.5,12.5) edge[-latex] (16,12.5);

\draw (0.75,5) edge[-latex] (2.25,5);
\node[] at (3,6.5) {$\mapsto$};
\draw (3.5,5) edge[-latex] (5,5);
\draw (3.5,6) edge[-latex] (5,8);
\draw (3.5,7) edge[-latex] (5,8);
\draw (3.5,8) edge[-latex] (5,8);

\draw (6.25,5) edge[-latex] (7.75,7);
\node[] at (8.5,6.5) {$\mapsto$};
\draw (9,5) edge[-latex] (10.5,7);
\draw (9,6) edge[-latex] (10.5,8);
\draw (9,7) edge[-latex] (10.5,8);
\draw (9,8) edge[-latex] (10.5,8);

\draw (11.75,5) edge[-latex] (13.25,6);
\node[] at (14,6.5) {$\mapsto$};
\draw (14.5,5) edge[-latex] (16,6);
\draw (14.5,6) edge[-latex] (16,6);
\draw (14.5,7) edge[-latex] (16,8);
\draw (14.5,8) edge[-latex] (16,8);

\draw (0.75,1.5) edge[-latex] (2.25,2.5);
\node[] at (3,2) {$\mapsto$};
\draw (3.5,0.5) edge[-latex] (5,2.5);
\draw (3.5,1.5) edge[-latex] (5,2.5);
\draw (3.5,2.5) edge[-latex] (5,3.5);
\draw (3.5,3.5) edge[-latex] (5,3.5);

\node[] at (8.5,2) {$\mapsto$};
\draw (9,0.5) edge[-latex] (10.5,3.5);
\draw (9,1.5) edge[-latex] (10.5,3.5);
\draw (9,2.5) edge[-latex] (10.5,3.5);
\draw (9,3.5) edge[-latex] (10.5,3.5);
\end{tikzpicture}
\caption{The bijection $\alpha\mapsto\overline\alpha$ of $IC_3$ onto $C_4$}\label{fig:bijection}
\end{figure}

\end{document}